\documentclass[12pt]{amsart}

\usepackage{amstext,amssymb,amsmath,stmaryrd,amsthm}
\usepackage{geometry} 
\usepackage{color}
\usepackage{hyperref}
\makeindex

\numberwithin{equation}{section}

\newtheorem{theorem}{Theorem}[section]
\newtheorem{corollary}[theorem]{Corollary}
\newtheorem{lemma}[theorem]{Lemma}
\newtheorem{proposition}[theorem]{Proposition}

\usepackage{amssymb}
\usepackage{amsmath}
\usepackage{color}
\usepackage{graphicx}
\usepackage{comment}
\usepackage{float}
\usepackage{hyperref}

\usepackage{amssymb}
\usepackage{amsmath}
\usepackage{color}
\usepackage{graphicx}
\usepackage{comment}
\usepackage{float}
\usepackage{hyperref}

\def\N{{\mathbb N}}

\def\II{\textbf{II}}
\def\III{\textbf{III}}

\usepackage{amstext,amssymb,amsmath,stmaryrd,amsthm}
\usepackage{geometry} 
\usepackage{color}
\usepackage{hyperref}

\newcommand{\R}{\mathbb{R}}

\newcommand{\Z}{\mathbb{Z}}
\newcommand{\E}{\mathbb{E}}

\newcommand{\Ge}{\mathcal{L}}

\begin{document}

\title[modified log-Sobolev inequality for a PJMP.]{Modified log-Sobolev inequality for a compact  Pure Jump Markov Process with degenerate jumps.}

\author{Ioannis Papageorgiou}{ Ioannis Papageorgiou   \\  IME, Universidade de Sao Paulo }

 \thanks{\textit{Address:} Neuromat, Instituto de Matematica e Estatistica,
 Universidade de Sao Paulo, 
 rua do Matao 1010,
 Cidade Universitaria, 
 Sao Paulo - SP-  Brasil - CEP 05508-090.
\\ \text{\  \   \      } 
\textit{Email:}  ipapageo@ime.usp.br, papyannis@yahoo.com  \\
 This article was produced as part of the activities of FAPESP  Research, Innovation and Dissemination Center for Neuromathematics (grant 2013/ 07699-0 , S.Paulo Research Foundation); This article  is supported by FAPESP grant   (2017/15587-8) }
\keywords{modified log-Sobolev inequality,  brain neuron networks, Pure Jump Markov Processes, concentration, empirical approximations}
\subjclass[2010]{  60K35,  26D10,       60G99} 



\begin{abstract}
We study the modified log-Sobolev inequality   for a class of pure jump Markov processes that describe the interactions between brain neurons. In particular, we focus on a finite and compact  process with degenerate jumps inspired by the model introduced  by Galves and L\"ocherbach in \cite{G-L}. As a result, we obtain concentration properties for empirical approximations of the process.    
\end{abstract}
 
\date{}
\maketitle

\section{Introduction}

 We study properties of the model introduced by  Galves and L\"ocherbach  in \cite{G-L},   in order to    describe the interaction activity occurring between brain neurons. We focus in particular on finite networks of compact   neurons taking  values in the domain of the invariant measure. What is in particular interesting about this  jump process  is the  degenerated character of the  jumps,  in the sense that after a particle  spikes,  it then jumps to zero and therefore  loses its memory.    In addition,   the spike  probability of a specific  neuron at any time   depends on  its actual position at that time and so depends on  the past of the whole neural system since the last time   this neuron had a spike.

The   aim of the paper is to show the modified logarithmic Sobolev inequality  for  the model and consequently obtain empirical concentration properties. In \cite{H-P}    Poincar\'e type inequalities were    proven. There were two separate cases that were examined. At first,   the initial configuration was a general one, and then   the initial configuration belonged to the domain of the invariant measure.   In the current paper where we are restricted  exclusively to the case where the initial configuration belongs to the domain of the invariant measure we will obtain the stronger modified log-Sobolev inequality, of the following form
 \begin{align}\label{tLogSob}P_t \left( f\log \frac{f}{P_t f}\right) \leq C(t) P_t\left( \frac{\Gamma(f,f)}{f}\right),\end{align}
for   the associated semigroup   $P_t$. In that way, we show that when the process enters  at the invariant domain,   despite the degenerate nature of its jumps, it behaves similar  to a non degenerate jump process.

As a result of the modified log-Sobolev inequality,  for any neuron $j$ and times   $t_1<...<t_n<T$      for some $T> 0$, we can obtain concentration inequalities for empirical approximations of the process as the ones shown bellow 
  \[ P \left( \left\vert \frac{\sum_{k=1}^{n}f(X^{j}_{t_k})}{n} - \frac{\sum_{k=1}^{n}\E [f(X^{j}_{t_k})]}{n}  \right\vert \geq  \epsilon   \right)\leq  De^{-\epsilon n}  . \]
   These imply that the empirical approximation converges exponential fast to the process as the number of the observables $n$ goes to infinity.
  
 In the next section, we present the model that describes the  neuroscience framework of the problem.

\subsection{The model.}
We want to model the action potential, called spike, of the membrane potential of a neuron. In relation to this spiking activity, there are two important features. The first is the degenerate nature of the spiking, that relates to the fact that  whenever a neuron spikes its membrane potential resets to zero. The second characteristic is that the probability of a neuron $j$ with membrane potential on  an actual position $x_t^j$ to spike at a given time $t$, depends on its position at this time, through its  intensity $\phi(x_t^j),$ where $\phi:\R_+ \to \R_+ $ is a given intensity function.
The interaction between the neurons  occurs by all   the post-synaptic neurons   $i$  receiving an additional amount of membrane potential $W_{j\rightarrow i}$ from the neuron $j$ that spiked. In the current work we consider the case of  pure jump Markov process,   abbreviated as  PJMP, where inactivity occurs, e.g lack of a drift, between two consecutive  spikes.

The spiking activity of an individual neuron, can be modeled by a single point process as   in  \cite{C17}, \cite{D-L-O},  \cite{D-O}, \cite{G-L}, \cite{H-R-R} and \cite{H-L}. In these papers the emphasis is put on describing the spiking time. Here however, we focus on modelling the interactions occurring between the neurons in the network through spikes, as was done in \cite{An}, \cite{H-K-L}, \cite{Davis84}, \cite{Davis93}, \cite{C-D-M-R},  \cite{L17}, \cite{PTW-10} and \cite{ABGKZ}. 

To do this, for a network comprising of $N>1$ neurons, we consider  the Markov  process $X_t=(X_t^1,...,X_t^N),$ representing the membrane potential of each neuron at time $t \in \R_+$.  Then, for every $t\geq 0$ and $i=1,...,N$,   $X_t$ solves the following equation 
\begin{align*}
X^{ i}_t = & X^{i}_0  -  \int_0^t \int_0^\infty 
X^{ i}_{s-}  1_{ \{ z \le  \phi ( X^{ i}_{s-}) \}} N^i (ds, dz) \\
 &+    \sum_{ j \neq i } W_{j \to i} \int_0^t\int_0^\infty  1_{ \{ z \le \phi ( X^{j}_{s-}) \}} 1_{ \{  X^{ i}_{s-} \leq m-W_{j \to i} \}}N^j (ds, dz),
\nonumber
\end{align*}  
where $(N^i(ds, dz))_{i=1,\dots,N}$ is a family of i.i.d.  Poisson random measures on $\R_+ \times \R_+ $ with  intensity measure $ds dz$. 

For any test function $ f : \R_+^N \to \R $  and $x \in [0,m]^N$ the generator of the process is given by \begin{equation}\label{eq:generator0}
\Ge f (x ) = \sum_{ j = 1 }^N \phi (x^j) \left[ f ( \Delta_j ( x)  ) - f(x) \right]
\end{equation}
where
we have denoted\begin{equation}\label{eq:delta1}
(\Delta_j (x))_i =    \left\{
\begin{array}{lcl}
x^{i} +W_{j \to i}  & i \neq j    \text{ and }  x^{i} +W_{j \to i} \leq m   \\  
x^{i}   & i \neq j    \text{ and }  x^{i} +W_{j \to i}  > m    \\
0 & i = j 
\end{array}
\right\}
\end{equation}
for some $m>0$ and weights $W_{j\rightarrow ij}>0$. Furthermore, we also assume that for some strictly positive constant  $\delta$,  the intensity function satisfies the following conditions:
\begin{equation}\label{phi2}
\phi(x)\geq \delta.
\end{equation}

   \subsection{Main results.}
 
 For simplicity, we will widely use the following convention. For the expectation of a function $f$ with respect to a measure  $\nu$ we will write \[\nu (f)=\int fd\nu.\]   
   We consider  a Markov process $(X_t)_{t\geq 0}$ which is described by the  Markov semigroup $P_t f(x)=\E^x(f(X_t))$ and $\Ge$ the associated infinitesimal generator.  
   
   We define $\mu$ to be the invariant measure for the semigroup $(P_t)_{t\geq 0}$ if and only if 
   \[\mu P_t =\mu.\]    
Define $D$    the domain of the invariant measure $\mu$,  that is 
\[D=\left\{x \in (\Z\cap [0,m])^N:  \mu (x)>0 \right\}.\]   Furthermore, we define the so called "carr\'e du champ" operator by:
   \[\Gamma (f,g):=\frac{1}{2}(\Ge (fg)-f\Ge g- g\Ge f).\]
For the  PJMP  process  defined as in (\ref{eq:generator0})-(\ref{eq:delta1}) 
  we then have
 \begin{align*}\Gamma (f,f)= \frac{1}{2}(\sum_{ i = 1 }^N \phi (x^i) \left[ f ( \Delta_i ( x)  )- f(x) \right]^2).
 \end{align*}

We are interested in studying the modified log-Sobolev inequality for the semigroup $P_t$ on a discreet setting (see   \cite{SC}, \cite{D-SC}, \cite{W-Y}, \cite{A-L}  and \cite{Chaf}).   In \cite{H-P} a Poincar\'e type   inequality was  shown  for the semigroup $P_t$ of the  bounded process (\ref{eq:generator0})-(\ref{eq:delta1}) for general initial configurations. In the current paper, we study again bounded neurons (that is $m<\infty$), but this time we focus exclusively  on the   case where the initial configuration belongs  on the domain of the invariant measure. Restricting the domain of the  process,  allows to  strengthen the results for the semigroup.   The  method that we use, is based on the so called semigroup method which is used to prove log-Sobolev and Poincar\'e inequalities for the semigroup $P_t$  (see \cite{A-L} and \cite{W-Y}), usually with a constant that depends on time $t$.   

Then, we study concentration inequalities for empirical approximations of the model. Although, in principle we use technics that relate modified log-Sobolev inequalities to concentration properties (see \cite{Ta1}  and \cite{Ta2}, \cite{Bo-Le}, \cite{Led} and \cite{Led0}), in order to obtain the empirical concentration inequalities, we actually  need to extend the results of the modified log-Sobolev inequality to cylindrical functions. Although, neither the initial modified inequality  obtained, nor the one for the cylindrical functions has the standard form (they both include the carr\'e du champ beyond one jump), we still manage to obtain the desired concentration properties. 

Before we proceed with the presentation of the    results we will clarify    a distinction on the dual nature of the initial configuration from which  the process may  start. This classification is based on the return probability to the initial configuration. We recall that the main mechanism of the dynamics dictates that the  membrane potential  of every neuron  lies  within some positive compact set and that whenever a neuron    spikes, every other neuron  jumps some length  up, while the only  movement downwards that it  can do is to fall to zero when and only  it spikes. Furthermore, in between spikes the neurons stay still. That implies that   there is a finite number of possible configurations to which the membrane potential of the neural system  can return after every neuron has spiked at least one time. This is the domain of the invariant measure $\mu$ of the semigroup $P_t$. As a result, whenever an  initial configuration   does not belong to the domain of the invariant measure, after the process has entered the invariant  domain  it can never return back to the initial configuration. In \cite{H-P} the focus was on general   initial configurations that were allowed not to    belong to the 
 domain of the invariant measure. Examples of such configurations are easy to construct. One can for instance think of configurations where $x^{i}=x$ for all $i=1,...,N$, or $x^{i}$ not being an analogue of $W_{j\rightarrow i}$. In the current work however, we restrict ourselves exclusively to the cases where the initial configuration, and so, every configuration, belongs in the domain of the invariant measure.

The first   result of the paper related to the modified   log-Sobolev inequality follows.
\begin{theorem}\label{thmCon} Assume the PJMP as described  in (\ref{eq:generator0})-(\ref{phi2}).  Then, for any $x\in D$,  we have the following type of modified log-Sobolev inequality\begin{align*}P_t\left( f(x)\log \frac{f(x)}{P_t f(x)}\right) \leq  & \delta(t)P_{t}\left(\frac{\Gamma(f,f)(x)}{f(x)}\right) +  \delta(t)\sum_{j=1}^N  P_{t}\left(\frac{ \Gamma(f,f)(\Delta_j(x))}{ f(x)}\right)+  \\  &+ \delta(t) \sum_{i,j=1}^N    P_{t}\left(\frac{ \Gamma(f,f)(\Delta_i(\Delta_j(x)))}{ f(x)}\right)
\end{align*}
where $\delta(t)$ is an increasing polynomial  of degree three, for which $\delta(0)=0$.\end{theorem}
The modified log-Sobolev inequality presented on the  theorem, is not very different from the standard form (\ref{tLogSob}) of the modified log-Sobolev inequality. At first it should be noticed that it implies the same concentration properties, as presented in Proposition \ref{Pcon}. Furthermore, for a special class of   functions it is equivalent to the standard modified log-Sobolev inequality studied for example in  \cite{Chaf} and \cite{A-L}, as shown on the following corollary.

 \begin{corollary} \label{corP}For any $i\in\{1,...,N\}$, define $f_i:[0,m]^N\rightarrow \R$ a function that depends only on  $x^{i}$. Assume the PJMP as described  in (\ref{eq:generator0})-(\ref{phi2}), with equal weights $w_{ij}=w,\ \forall 1\leq i,j\leq N$, for some $w\in (0,+\infty)$.  If for every $i\in\{1,...,N\}$, $f_i$  is either decreasing and convex, or increasing and concave, then for every $x\in D$, 
 \begin{align*}P_t\left( f_i(x)\log \frac{f_i(x)}{P_t f_i(x)}\right) \leq   \zeta(t)P_{t}\left(\frac{\Gamma(f_i,f_i)(x)}{f_i(x)}\right) \end{align*}
  where $\zeta(t)$ is a  polynomial  of degree three. 
 \end{corollary}
 This    modified log-Sobolev inequality  describes one neuron in the context of the whole system.
 
 It should be noted that the  main hindrance in obtaining a log-Sobolev inequality,  is  down to the degenerate character of the jump process under study, since the loss of memory of the spiking neuron  does not allow the     translation property 
$$\E^{x+y}f(X_t)=\E^{x}f(X_t+y)$$
used  in \cite{W-Y} and \cite{A-L} to show the relevant inequalities. The absence of the translation property   implies that the  inequalities $\Gamma (P_tf,P_tf)\leq P_t\Gamma (f,f)$ and $\sqrt{\Gamma (P_tf,P_tf)}\leq P_t\sqrt{\Gamma (f,f)}$ that are used to show Poincar\'e and log-Sobolev inequalities respectively  do not hold. This is directly related  with the $\Gamma_2$ criterion  (see \cite{Ba1} and \cite{Ba2}) which provides log-Sobolev and  Poincar\'e inequalities   (see also \cite{A-L}). Still, a weaker property shown here proves the modified log-Sobolev inequality of the theorem.  In that way, we see that despite the degenerate character of the process due to it's discontinuous jumps, when it enters the domain of the invariant measure, it does not behave very different from a non degenerate process which satisfies the typical sweeping out relation $\Gamma (P_tf,P_tf)\leq P_t\Gamma (f,f)$ and  the $\Gamma_2$ criterion. This is further testified by the concentration properties obtained bellow, as well as by the classical modified log-Sobolev inequality of Corollary \ref{corP}.  Still, both the higher than order  one   time constant and the additional two jump terms, highlight how much complicated this process is.  

As a direct result of the modified log-Sobolev inequality of Theorem \ref{thmCon} we obtain the following  concentration inequality. 
\begin{proposition}\label{Pcon}Assume the PJMP as described  in (\ref{eq:generator0})-(\ref{phi2})  and that $f_i$ is Lipschitz continuous functions that depend only on $x^{i}$,   with Lipschitz constant $1$. Consider $T>0$.  Then, for the function    $f(x)=\sum_{i=1}^{N}f_i(x^{i})$, $x\in D$, the following concentration inequality holds \begin{align*}  P\left(\left\vert f(X_{t})-\E^{x} f(X_t)\right\vert \geq  r \right)\leq Q e^{-r^{2}}  ,  \end{align*} 
for any $t\leq T$ and  a positive constant $Q$ that depends on $T$.
\end{proposition}
It should be noted that the  above concentration inequality is the concentration property that is typically derived from the standard modified log-Sobolev inequality. 

Concentration properties like this one are of course more  interesting   in the context of  non-bounded random variables, and less in the   context of finite many bounded valued random variables as in Proposition \ref{Pcon}, although in the context of the current paper they highlight the closeness between the modified inequality of Theorem \ref{thmCon} and the standard inequality of Corollary \ref{corP}. 

More interesting are concentration inequalities for empirical measures, as shown in the   following theorem. These concentration estimates show that the empirical  estimates stabilize exponentially fast as the number of observables  goes to infinity.
 \begin{theorem}\label{empthm} Consider some $T>0$ and $i\in\{1,...,N\}$. Assume $f:\R\rightarrow \R$ a Lipschitz function, with Lipschitz constant 1, and a     sequence of times $0=t_1<t_2<...<t_n\leq  T$,  such that \begin{align}\label{condit1}\sum_{k=1}^{\infty}  \delta(t_k-t_{k-1})  3^{k}\sum_{r=1}^{k+1} (Nd)^{r+4}<\infty,\end{align}
for  $\delta(t)$ as in Theorem \ref{thmCon} and a constant  $d=\frac{M (\max_ {ij}w_{ij})^{2}}{2}   e^{2 m}$. Then,   for any $x\in D$, we have 
  \begin{align*}  P\left(\left\vert \frac{\sum_{k=1}^{n}f(X^{i}_{t_k})}{n}-\frac{\sum_{k=1}^{n} \E^{x}[f(X^{i}_{t_k})]}{n}\right\vert \geq  \epsilon  \right)\leq G e^{-\epsilon n}   \end{align*} 
 where $G$ a strictly positive constant.
 \end{theorem}
 One should notice that the concentration approximation of the theorem, provides a measure of intrinsic proximity to equilibrium, since as  $n$ goes to infinity, the  $\frac{\sum_{k=1}^{n}\E^{x}[f(X^{i}_{t_k})]}{n}$ converges a.s, by the Statistical Ergodic Theorem (see \cite{Yo}),  to the  expectation of   $f$ with respect to the invariant measure $\mu (f)$.  
 Since $\delta(t)$ is an increasing polynomial of order three with    $\delta(0)=0$, in order to have condition  (\ref{condit1}) satisfied, we just need the distance of consequential times to decrease in such a way that 
 \[\delta(t_k-t_{k-1}) \leq (\frac{2}{3})^{k} \left( \sum_{r=1}^{k+1} (Nd)^{r+4 } \right)^{-1}.\]
  From a technical point of view, contrary to the concentration inequalities of Proposition \ref{Pcon}, the empirical concentration inequalities present the challenge of involving non bounded quantities, since they involve concentrations of      sums $\sum_{k=1}^{n}f(X^{i}_{t_k})$, for $n$ going to infinity.  One should observe that the empirical  concentration property does not follow directly from Proposition \ref{Pcon}, as will be  discussed in section \ref{secCon}, but from an extension of the modified inequality for cylindrical functions.

A few words about the structure of the paper. As already mentioned, the absence of the translation property, poses difficulties in obtaining the appropriate sweeping out relations of the form  $\Gamma(P_t f, P_t f)\leq P_t \Gamma (f,f)$, to prove   modified log-Sobolev inequalities. Thus, in the first section   \ref{local} of the paper where we show the modified log-Sobolev inequalities, we focus on obtaining an alternative weaker sweeping out inequality (Lemma \ref{lastLem}). Since, the process we study is characterized by degenerate jumps, the inequality we obtain involves additional terms that include the carr\'e du champ after the first jump. As a result, the modified inequality we obtain in Theorem  \ref{thmCon}  also includes some additional  terms that involve the carr\'e du champ after the first jump as well. 

These additional terms however, do not alter in essence    the inequality, since as we show in Proposition \ref{Pcon}, it implies the same concentration properties with the typical modified log-Sobolev inequality (\ref{tLogSob}). The proof of these concentration properties are presented at the very end of the paper in section \ref{Pcon2}.  Furthermore, we see that for a class of functions the inequality presented in Theorem  \ref{thmCon} can be reduced to that of the  typical form  (\ref{tLogSob}) as presented in 
Corollary \ref{corP}. The proof of this result follows   the proof of the main modified log-Sobolev inequality, at section \ref{proofcor}.   

The empirical concentration inequality   of Theorem \ref{empthm} is presented in \ref{proofEmpCon}.  Since the property refers to cylindrical multi-times functions, to show the concentration properties we first extend some of the properties  obtained in \ref{local}   to cylindrical functions, as for instance is a generalised sweeping out relation shown in Lemma \ref{InducLem}. 

 \section{proof of the modified log-Sobolev inequality. \label{local}}
 
 We start by showing some technical results.

\subsection{Technical results.}
We start by  showing properties of  the jump probabilities of the degenerate PJMP processes. 
Our process is restricted on  the compact  domain
$D':=\{x \in \R_+^N: x^{i} \leq m , \, 1 \leq i \leq N \} $. Since we exclusively study configurations on the domain of the invariant measure $\mu$, that is $D=\left\{ x\in D': \mu (x)>0  \right\}$,   we write $D$, for the elements of  $D'$ that belong to the domain  of the invariant measure.  Since each of the finite many neurons can visit only a finite number of positions, by standard arguments of  finite dimensional compact  discreet Markov Chains following from the Perron-Frobenius theorem (see for instance \cite{SC} and \cite{No}), we conclude that the invariant measure $\mu$ exists and is unique. Following the same argumentation, if we denote the probability the process starting from $x$ to be at $y$ after time $t$ by
$$\pi_t(x,y):=P_x(X_t=y),$$
then 
\[\lim_{t\rightarrow \infty}\pi_t(x,y)=\mu (y)\geq e>0,\] 
for some $e>0$ uniformly on $x$ and $y$.
If we define   the set of reachable positions of the process starting from $x$ after time $t$ as $D_x:=\{y \in D, \pi_t(x,y) >0\}$, 
  then one should observe that since there is not movement between two consecutive spikes, $D_x$ is finite. 

For any time $s \in \R_+$ and    $x \in D,$ we denote by $p_s(x)$ the probability that starting at time $0$ from position $x,$ the process has no jump in the interval $[0,s]$. Then, if we denote $\overline{\phi}(x)=\sum_{j \in I} \phi(x^{j})$, we have
\[p_s(x)= e^{-s \overline{\phi}(x)}.
\] 
Furthermore, for a given neuron $i \in I$  denote by $p_s^i(x)$ the probability that in the interval $[0,s]$ only the  neuron $i$ spikes, and it does exactly one time. Then,  for every $x\in D$  s.t. $\overline{\phi}(\Delta^i(x)) \neq \overline{\phi}(x),$ we compute
\[
p_s^i(x)= \int_0^s \phi(x^{i}) e^{-u\overline{\phi}(x)}e^{-(s-u)\overline{\phi}(\Delta^i(x))}du 
 = \frac{\phi(x^{i})}{\overline{\phi}(x) - \overline{\phi}(\Delta^i(x))} \left( e^{-s\overline{\phi}(\Delta^i(x))} - e^{-s\overline{\phi}(x) } \right) ,
\]
while \[p^i_s(x)=s \phi(x^{i})e^{-s \overline{\phi}(x)}\] when $ \overline{\phi}(\Delta^i(x)) = \overline{\phi}(x) $. One should observe that   $p_s^i(x)$ is  continuous, strictly increasing on $(0, t_0(i,x))$ and strictly decreasing on $(t_0(i,x),+\infty)$, for 
$t_0(i,x)=   
\frac{ln\left( \overline{\phi}(x) \right) -ln \left( \overline{\phi}(\Delta^i(x)) \right)}{\overline{\phi}(x) - \overline{\phi}(\Delta^i(x))} $ when $ \overline{\phi}(\Delta^i(x)) \neq \overline{\phi}(x)$ and $t_0(i,x)=
\frac{1}{\overline{\phi}(x)}$ when $\overline{\phi}(\Delta^i(x)) = \overline{\phi}(x) $. The following two lemmata follow partly technics applied   in  \cite{H-P} to show  similar  bounds, only that in the current paper, taking advantage of the restriction to configurations on the domain of the invariant measure, we obtain stronger results. 

Since by the construction of the process, both the number of neurons  and the  cardinality of $D$ are bounded, we can define
\[t_0:=\min_{x\in D, i\in \{1,...,N\}}t_0(i,x),\]
which is strictly positive.
 \begin{lemma}\label{frac1} 
 Assume the PJMP as described in (\ref{eq:generator0})-(\ref{phi2}).There exists a positive constant $C_{1.1}$  such that for every  $u\leq t$
 \[
 \frac{\pi_u (x,y)}{\pi_t(x,y)} \leq C_{1.1} .  \]
  \end{lemma}

  \begin{proof}
Since $D$ is finite,   there exists    a constant $e>0$, such that for every $x\in   D$, one has $\mu (x)>e>0$. Since, $\lim_{t\rightarrow \infty }\pi_t(x,y)=\mu(y)$ for every $x,y \in  D$  we conclude that
\begin{align} \label{timeHt}\exists\ \hat t>0:\forall t\geq \hat t,  \pi_t(x,y)>e, \forall x,y\in D.\end{align} We can then write 
\[  \frac{\pi_u (x,y)}{\pi_t(x,y)} \leq \frac{1}{e},\]
which proves the bound for every $t\geq \hat t$.
It remains to show the same result for   the case  $t\leq\hat t$. Since $u\leq t$ we can write 
\[\pi_t(x,y)\geq \pi_u(x,y)p_{t-u}(y)= \pi_u(x,y)e^{-(t-u) \overline{\phi}(y)}\geq \pi_u(x,y)e^{-\hat t N\hat \phi},\]
where above we have denoted $\hat \phi:=\sum_{x\in[0,m]}\phi(x)$. This  implies
  \[ \frac{\pi_u(x,y)}{\pi_t(x,y)} \leq e^{ \hat t N\hat \phi}\]
  for every $t\leq \hat t$.

 \end{proof}

 \begin{lemma}\label{frac2} 
 Assume the PJMP as described in (\ref{eq:generator0})-(\ref{phi2}).There exists a positive constant $C_{1.2}$  such that for every  $u\leq t$
    \[
  \frac{\pi_u^2(\Delta^i(x),y)}{\pi_t(x,y)} \leq C_{1.2} 
  \]
  for every $t\geq t_0$, as well as, for every $t\leq t_0$   $\forall y \in D_x\setminus \{\Delta^i(x)\}$.
 \end{lemma}

  \begin{proof} For  $\hat t$ as in (\ref{timeHt}), we distinguish three separate cases:

\  \ \ \ \   \   \   \   \   \   \   \   \  \     \     \ (A) $t\geq \hat t$, (B) $t_0<t\leq \hat t$ and (C) $t \leq  t_0$.

A) At first we examine   the case $t\geq \hat t$. As in the previous lemma,   for every $t\geq \hat t$, we have  $\pi_t(x,y)>e$, which directly  leads to the following bound
     \[
  \frac{\pi_u^2(\Delta^i(x),y)}{\pi_t(x,y)} \leq \frac{1}{e}
  \]
  for every $t\geq \hat t$.

B)  We now study the case  $t_0<t\leq \hat t$. Here we also distinguish over separate subcases (B1) and (B2). 
  
B1) If $ p_{t_0}^i(x)\leq \pi_{t-t_0}(\Delta^i(x),y)$, we can then write  \[\pi_{t}(x,y) \geq p_{t_0}^i(x)\pi_{t-t_0}(\Delta^i(x),y)\geq ( p_{t_0}^i(x))^2,\]
in order to bound the denominator. The numerator can be bounded by $\pi_u (\Delta^i(x),y)\leq 1$. This gives the bound
\[\frac{\pi_u^2(\Delta^i(x),y)}{\pi_{t}(x,y)} \leq \frac{1}{(p_{t_0}^i(x))^2}.
\]
 
B2) Now consider   $\pi_{t- t_0}(\Delta^i(x),y) < p_{ t_0}^i(x)$  and  recall that $p_s^i(x)$ as a function of $s$ is continuous, strictly increasing on $(0,  t_0)$ with $p_0^i(x)=0$. Also, $\pi_{t-s}(\Delta^i(x),y)$ as a function of $s$ is continuous and takes value $\pi_{t}(\Delta^i(x),y) > 0$ for $s=0$.  We conclude that  there exists $s_* \in (0, t_0)$ such that $p_{s_*}^i(x) = \pi_{t-s_*}(\Delta^i(x),y).$

Once more we will consider two subcases. 

B2.1) At first assume that  $u\leq t-s_* $. Then  
\[\frac{\pi_u^2(\Delta^i(x),y)}{\pi_{t}(x,y)} \leq \frac{\pi_u^2(\Delta^i(x),y)}{p_{s_*}^i(x)\pi_{t-s_*}(\Delta^i(x),y)}=\frac{\pi_u^2(\Delta^i(x),y)}{ \pi_{t-s_*}^{2}(\Delta^i(x),y)}\leq C^2_{1.1} \]from  Lemma \ref{frac1}.

B2.2) Now we consider the case where  $u\geq t-s_* $. We can write
\begin{equation}\label{fract}
\frac{\pi_u^2(\Delta^i(x),y)}{\pi_{t}(x,y)} \leq \frac{\left(\pi_{u-s'}(\Delta^i(x),y)p_{s'}(y) + \sup_{z \in D} (1-p_{s'}(z)) \right)^2}{p_{s}^i(x)\pi_{t-s}(\Delta^i(x),y)}
\end{equation}
for any $  s'\in(0,u)$ and $s\in (0,t)$. If we choose  $s= s_*$ and $s'\geq 0$ s.t. $u-s'=t-s_*$ we get
 \[
\frac{\pi_u^2(\Delta^i(x),y)}{\pi_t(x,y)} \leq (p_{s'}(y))^2+2 p_{s'}(y) \frac{1- e^{-s_*N\hat \phi}}{p_{s_*}^i(x)} + \left( \frac{1- e^{-s_*N\hat \phi}}{p_{s_*}^i(x)} \right)^2,
\]where above we also use that since $u\leq t$ we have $s'\leq s_*$ and so  $ 1- e^{-s'N\hat \phi}\leq 1- e^{-s_*N\hat \phi}$. To bound the right hand side, we need to bound $\frac{1- e^{-s_*N\hat \phi}}{p_{s_*}^i(x)}$. Since $s_*\leq t_0$, we obtain
\[
\frac{1- e^{-s_*N\hat \phi}}{p_{s_*}^i(x)} \leq \left\{
\begin{array}{ll}
\frac{e^{ t_0N\phi(m)}}{\delta} \, \frac{\left( \overline{\phi}(x) - \overline{\phi}(\Delta^i(x))\right) \left( 1-e^{-s_*N\hat \phi}\right)}{1-e^{-s_*\left( \overline{\phi}(x) - \overline{\phi}(\Delta^i(x))\right)}} & \, \mbox{if } \,  \overline{\phi}(\Delta^i(x)) \neq \overline{\phi}(x) \\ \\
\frac{e^{ t_0N\hat \phi}}{\delta} \frac{1-e^{-s_*}}{s_*} & \, \mbox{if } \, \overline{\phi}(\Delta^i(x)) = \overline{\phi}(x) 
\end{array}
\right\},
\]
where above we also used the lower bound   $\phi(x) \geq \delta$ from condition (\ref{phi2}). One notices that when $s_*$ goes to zero we obtain a bound that depends on $t_0$.  Since the right hand side is bounded uniformly for every $s_*\leq t_0$ we obtain the desirable bound.

C) To finish the proof, it remains to consider the case  where $t \leq  t_0$ and $y \neq \Delta^i(x)$. We will use  \eqref{fract} again.  Since $\pi_{t-s}(\Delta^i(x),y)$ is continuous as a function of $s$ and takes values $\pi_{t}(\Delta^i(x),y) > 0$ and $\pi_0(\Delta^i(x),y) = 0$ respectively for $s=0$ and $s=t$, we deduce that there exists $s_* \in (0,t) \subset (0,t_0)$ such that $p_{s_*}^i(x) = \pi_{t-s_*}(\Delta^i(x),y)$ and we are back in the previous case, and  so the   desirable bound follows similarly to (B2.2).  
 \end{proof}

\subsection{modified log-Sobolev inequality}\label{proofCon1}

We start by showing some useful lemmata that will be used to bound the entropy of the semigroup.

 \begin{lemma}\label{bound1}  Assume the PJMP as described in  (\ref{eq:generator0})-(\ref{phi2}). If $t-s\geq t_0$, then for every $x\in D$
   \begin{align*} \II_1:=\left(\int_{0}^{t-s}(\E^{\Delta_i ( x) }-\E^x)( \Ge f(X_{u})) du\right)^2 \leq &  \\ 2 (t-s)^2 MC_1 & P_{t-s}\left(\frac{\Gamma (f , f)(y)}{f(y)}\right) P_{t-s}(f(x))  . 
 \end{align*}
\end{lemma}
\begin{proof}
 For $\pi_t(x,y)$ being the probability kernel of $\E^x$, we have  \[P_tf(x)=\E^x (f(X_t))=\sum_y \pi_t(x,y) f(y).\] Then we can write
 \begin{align*}\nonumber
\II_1 =& \left( \int_{0}^{t-s}\sum_{y\in D}\left(\pi_u(\Delta_i (x),y)- \pi_u(x,y)\right )   \Ge f(y)du\right)^2 
 \\  \leq& 2(t-s)\int_{0}^{t-s}\underbrace{ \left( \sum_{y\in D}  \pi_u(\Delta_i (x),y) \Ge f(y)\right)^2}_{:=\Psi_1}+\underbrace{\left( \sum_{y\in D}\pi_u(x,y) \Ge f(y)  \right)^2}_{:=\Psi_2}du \end{align*}
  where in the last bound we used Jensen's inequality to pass the square inside the integral. Since for every $z\in D$ the number of sites that can be visited are finite, define  $d=\max_{z\in D}\vert D_z \vert$.  If we use the Cauchy-Schwarz  inequality to bound the square of the sum we obtain
   \begin{align*}  \Psi_1=& \left( \sum_{y\in D} \pi_u(\Delta_i (x),y)\frac{ \Ge f(y)}{f(y)^{\frac{1}{2}}}f(y)^{\frac{1}{2}}\right)^2  \\  \leq &d^2 \left( \sum_{y\in D}  \pi_u^2(\Delta_i (x),y)\frac{ (\Ge f(y))^2}{f(y)}\right)\left( \sum_{y\in D}  \pi^2_u(\Delta_i (x),y)f(y)\right) \\  = &d^2  \left( \sum_{y\in D} \pi_{t-s}(x,y)\frac{ \pi_u^2(\Delta_i (x),y)}{\pi_{t-s}(x,y)}\frac{ (\Ge f(y))^2}{f(y)}\right)\left( \sum_{y\in D}  \pi_{t-s}(x,y)\frac{ \pi_u^2(\Delta_i (x),y)}{\pi_{t-s}(x,y)}f(y)\right) . \end{align*}
Since $u\leq t-s$ and $t-s\geq t_0$, we can now use Lemma  \ref{frac2}  to bound  the two fractions  
 \begin{align*}  \Psi_1\leq C_{1.2}^2d^2  \left( \sum_{y\in D} \pi_{t-s}(x,y) \frac{ (\Ge f(y))^2}{f(y)}\right)\left( \sum_{y\in D}  \pi_{t-s}(x,y) f(y)\right) . \end{align*}
Denote $M:=\sup_{x\in D}\left(\sum_{ i = 1 }^N\phi(x^{i})+1\right)^2$. About the square of  the generator $\Ge(f)$ of $f$,  we can write 
  \begin{align*}(\Ge(f)(y))^2=&(\sum_{ i = 1 }^N\phi(y^{i}))^2\left( \sum_{ i = 1 }^N\frac{ \phi (y^i)}{\sum_{ i = 1 }^N\phi(y^{i})} \left[ f( \Delta_i ( y)  ) - f(y) \right]\right)^2\leq    M\Gamma (f , f)(y)\end{align*} 
  where above we first divided with the   normalisation  constant $\sum_{ i = 1 }^N \phi (\Delta_i ( x)^i)$, since $\phi(x)\geq \delta$, and then used  Jensen's inequality to pass the square inside the sum. 
  Putting everything together, we get 
   \begin{align*}  \Psi_1\leq C_{1.2}^2Md^2 P_{t-s}  \left(\frac{ \Gamma (f , f)(y)}{f(y)}\right)P_{t-s}( f(y)) . \end{align*}
We now compute $\Psi_2$. We will use again Cauchy-Schwarz, but this time  for the measure $P_{t-s}$. For this we will write  
\begin{align*}\Psi_2=&\left( \sum_{y\in D}\pi_u(x,y) \Ge f(y)  \right)^2 \\ =& \left( \sum_{y\in D}\pi_{t-s}(x,y)\frac{ \pi_u(x,y)}{\pi_{t-s}(x,y)}\frac{ \Ge f(y)}{f(y)^\frac{1}{2}}f(y)^\frac{1}{2}  \right)^2  \\   \leq & 
\left( \sum_{y\in D}\pi_{t-s}(x,y)\left(\frac{ \pi_u(x,y)}{\pi_{t-s}(x,y)}\right)^2\frac{( \Ge f(y))^2}{f(y)}   \right)
\left( \sum_{y\in D}\pi_{t-s}(x,y)f(y)   \right).
\end{align*}  We can bound $(\Ge(f ))^2$ as we did in  the computation of $\Psi_1$ and bound  the fraction from Lemma \ref{frac1}, to get 
\begin{align*}\Psi_2   \leq  
C_{1.1}^2 Md^2  P_{t-s} \left( \frac{\Gamma(f,f)(y)}{f(y)}   \right)
P_{t-s}(f(y)).
\end{align*} 

 One should notice that the upper bounds of $\Psi_1$ and $\Psi_2$ do not depend on the integration variable $u$ appearing in $\II_1$. So, if we put everything together we finally  obtain 
  \begin{align*} \II_1  \leq 2 (t-s)^2 MC_1  P_{t-s}\left(\frac{\Gamma (f , f)(y)}{f(y)}\right) P_{t-s}(f(x))  \nonumber\end{align*}where $C_1=d^2 (C^2_{1.1}+C^2_{1.2})$.
\end{proof}
\begin{lemma}\label{bound2}  Assume the PJMP as described in  (\ref{eq:generator0})-(\ref{phi2}). Then, for  $t-s<t_0$, 
 \begin{align*}\II_2 : =& \left(\int_{0}^{t-s}(\E^{\Delta_i ( x) }-\E^x) \Ge f(X_{u}) du\right)^2 \leq  8t_0^2 M      \Gamma(f,f)(\Delta_i(x)) + \\  &+ 4 t_0 ^2 MC_1  P_{t-s}\left(\frac{\Gamma (f , f)(y)}{f(y)}\right) P_{t-s}(f(x)) .
 \end{align*}
\end{lemma}
\begin{proof} We will work as in the previous lemma. Since $t-s< t_0$,   the bounds from Lemma \ref{frac2}   do not hold for all $y\in D$ and so we will break the sum in two parts as shown below.
 \begin{align}\nonumber
\II_2  \leq    &  \nonumber  2\underbrace{\left(\int_0^{t-s}( \pi_u(\Delta_i ( x),\Delta_i ( x)) -\pi_u(x,\Delta_i ( x)) ) \Ge f(\Delta_i ( x))) du\right)^2}_{:=\III_1} + 
\\   &    +2 \underbrace{  \left(\int_0^{t-s}(\sum_{y\in D, y\not=\Delta_i (x)}(\pi_u(\Delta_i (x),y)-\pi_u(x,y))\Ge f(y)) du\right)^2}_{:=\III_2}.  \label{pr2.4+1}
\end{align} 
 We first calculate the first summand  of (\ref{pr2.4+1}). We can write
\begin{align*}
\III_1     \leq     &  \nonumber  4\left(\int_0^{t-s}(\sum_{ j = 1 }^N \phi (\Delta_i ( x)^j)\vert f (\Delta_j ( \Delta_i ( x)))-f (\Delta_i ( x))\vert du\right)^2 \\ \leq  &  \nonumber  4t_0^2( \sum_{ j = 1 }^N \phi (\Delta_i ( x)^j))^2\left(\sum_{ j = 1 }^N \frac{\phi (\Delta_i ( x)^j)}{\sum_{ j = 1 }^N \phi (\Delta_i ( x)^j)}\vert f (\Delta_j ( \Delta_i ( x)))-f (\Delta_i ( x))\vert \right)^2   \end{align*}
where above we divided with the   normalisation constant $\sum_{ j = 1 }^N \phi (\Delta_i ( x)^j)$, since $\phi(x)\geq \delta$.  We can now    apply the Holder inequality on the sum,   so that 
  \begin{align*}\nonumber
\III_1    & \leq      
  4t_0^2M(\sum_{ j = 1 }^N \phi (\Delta_i ( x)^j) (f (\Delta_j ( \Delta_i ( x)))-f (\Delta_i ( x))^2)  \\ &=
  4t_0^2M\Gamma(f,f)(\Delta_i ( x)) .
\end{align*} 
 We now calculate the second summand  of (\ref{pr2.4+1}). For this term  we will work similar to Lemma \ref{bound1}.   
 \begin{align*}\nonumber
 \nonumber \III_2\leq & 2   t_0 \int_{0}^{t-s}\underbrace{ \left( \sum_{y\in D, y\not=\Delta_i (x)} \pi_u(\Delta_i (x),y) \Ge f(y)\right)^2}_{:=\Theta_1}+\underbrace{\left(\sum_{y\in D, y\not=\Delta_i(x)} \pi_u(x,y) \Ge f(y)\right)^2}_{:=\Theta_2}du. \end{align*}
Since when $y\in D,y\neq \Delta_i(x)$ the bounds from  lemmata \ref{frac1} and \ref{frac2}  still hold even when $t\leq t_0$, we can bound $\Theta_1$ and $\Theta_2$ exactly as we did in the previous lemma for $\Psi_1$ and $\Psi_2$ respectively, and so we  eventually  obtain
  \begin{align*} \III_2  \leq  4 t_0 ^2 MC_1  P_{t-s}\left(\frac{\Gamma (f , f)(y)}{f(y)}\right) P_{t-s}(f(x))   .  \nonumber\end{align*}
  Combining the bounds for $\III_1$ and $\III_2$ proves the lemma.  
  \end{proof}
Combining together Lemma \ref{bound1} and Lemma \ref{bound2} we get
 \begin{corollary}\label{techn} For  the PJMP as described in  (\ref{eq:generator0})-(\ref{phi2}), we have
 \begin{align*}\left(\int_0^{t-s}\left(\E^{\Delta_i ( x) }( \Ge f(x_{u}))-  \E^x(\Ge f(x_{u}))\right)du\right)^2 \leq &c    \Gamma(f,f)(\Delta_i(x)) + \\   +c(t-s)& P_{t-s}\left(\frac{\Gamma (f , f)(x)}{f(x)}\right) P_{t-s}(f(x)) 
 \end{align*}
 where $c= 8t_0^2 M $ and $c(t)=4 t_0 ^2 MC_1 + 2 t^2 MC_1 $. 
\end{corollary}
Next we show an   additive property for the semigroup,  when the semigroup is on the denominator. 
\begin{lemma}\label{neolemma}For  the PJMP as described in  (\ref{eq:generator0})-(\ref{phi2}), we have
\begin{align*}P_s \left(\frac{g(x)}{P_{t-s}f(x)}\right)\leq d(t-s)P_{t}\left(\frac{g(x)}{f(x)}\right) + d(t-s)\sum_{ j = 1 }^NP_{t}\left(\frac{g(\Delta_j(x))}{ f(x)}\right)\end{align*}
where $d(t)=2+8t^2M^2C_{1.1} $.\end{lemma}
\begin{proof}By Dynkin's formula
\[P_{t }g^{\frac{1}{2}}(x)=\E^x g^{\frac{1}{2}}(x_t)=g^{\frac{1}{2}}(x)+\int_0^{t}\E^x(\Ge g^{\frac{1}{2}}(x_u))du,\]
 we have
\begin{align}\label{nnn1}\frac{g(x)}{P_{t-s}f(x)}\leq 2 \frac{(\E^x g^{\frac{1}{2}}(x_{t-s}))^2}{P_{t-s}f(x)}+   2\frac{\left(\int_0^{t-s}\E^x(\Ge g^{\frac{1}{2}}(x_u))du\right)^2}{P_{t-s}f(x)}.\end{align}
 For the first term on the right hand side of (\ref{nnn1}), if we use the Cauchy-Schwarz inequality  we have
 \[\frac{(\E^x g^{\frac{1}{2}}(x_{t-s}))^2}{P_{t-s}f(x)}=\frac{(P_{t-s}g^{\frac{1}{2}}(x))^2}{P_{t-s}f(x)}\leq\frac{P_{t-s}\left(\frac{g(x)}{f(x)}\right)(P_{t-s}f(x))}{P_{t-s}f(x)}=P_{t-s}\left(\frac{g(x)}{f(x)}\right).\]
From the semigroup property  $P_sP_{t-s}=P_t,$ we get 
\begin{align}\label{neq0}P_s \left(\frac{(\E^x g^{\frac{1}{2}}(x_{t-s}))^2}{P_{t-s}f(x)}\right)\leq P_{t}\left(\frac{g(x)}{f(x)}\right).\end{align}
We will now compute the second  term in the right hand side of (\ref{nnn1}). From Jensen's inequality we have
\begin{align} \label{vnew2.5}\left(\int_0^{t-s}\E^x(\Ge g^{\frac{1}{2}}(x_u))du\right)^2\leq (t-s)  \int_0^{t-s}\left(\E^x\Ge( g^{\frac{1}{2}}(x_u))\right)^2du .
\end{align}
If we write
\begin{align*}\left(\E^x\Ge( g^{\frac{1}{2}}(x_u))\right)^2&= \left(\sum_{y\in D}\pi_u(x,y)\Ge( g^{\frac{1}{2}}(y))\right)^2\\ &=\left(\sum_{y\in D}\pi_{t-s}(x,y)\frac{\pi_u(x,y)}{\pi_{t-s}(x,y)}\frac{\Ge(g^{\frac{1}{2}}(y))}{f^{\frac{1}{2}}(y)}f^{\frac{1}{2}}(y)\right)^2,
\end{align*}from Cauchy-Schwarz  inequality and  Lemma \ref{frac1}  we bound
\begin{align*}\left(\E^x\Ge( g^{\frac{1}{2}}(x_u))\right)^2&
\leq\left(\sum_{y\in D}\pi_{t-s}(x,y)f(y)\right)\left(\sum_{y\in D}\pi_{t-s}(x,y)\frac{\pi^{2}_u(x,y)}{\pi^{2}_{t-s}(x,y)}\frac{(\Ge( g^{\frac{1}{2}}(y)))^{2}}{f(y)}\right)
\\ & \leq C^2_{1.1} \left(P_{t-s}f(x)\right)\left(\sum_{y\in D}\pi_{t-s}(x,y) \frac{(\Ge( g^{\frac{1}{2}}(y)))^{2}}{  f(y)}\right).
\end{align*}
  Furthermore,   if we   use once more   Cauchy-Schwarz  inequality and the bound $\sum_{ j = 1 }^N\phi(y^j)\leq M$, we have
\begin{align*}(\Ge(g^{\frac{1}{2}}(y)))^{2}&=\left(\sum_{ j = 1 }^N\phi(y^j)(g^{\frac{1}{2}}(\Delta_j(y))-g^{\frac{1}{2}} (y))\right)^2 \\  &\leq 2M^2\sum_{ j = 1 }^Ng  (\Delta_j(y))+2M^2g(y).\end{align*}   So we can bound  
\begin{align*}\left(\E^x\Ge( g^{\frac{1}{2}}(x_u))\right)^2
\leq &2 M^2 C^2_{1.1} \left(P_{t-s}  f(x)\right)\sum_{ j = 1 }^N P_{t-s}\left(\frac{g(\Delta_j(x))}{ f(x)}\right) +  \\  & +2M^{2}C^2_{1.1}\left(P_{t-s}f(x)\right)P_{t-s}\left(\frac{  g(x)}{ f(x)}\right) .
\end{align*}
From this and (\ref{vnew2.5}), we obtain the following bound for the second term on the right of (\ref{nnn1}) 
\begin{align*}  P_{s}\left(\frac{\left(\int_0^{t-s}\E^x(\Ge g^{\frac{1}{2}}(x_u))du\right)^2}{P_{t-s}f(x)}\right)\leq & (t-s)^2 2 M^2C^2_{1.1}  \sum_{ j = 1 }^N  P_{t}\left(\frac{g(\Delta_j(x))}{ f(x)}\right)\\  & +2(t-s)^2M ^2C^2_{1.1}P_{t}\left(\frac{ g(x)}{ f(x)}\right)   \end{align*}
 where once more we used that $P_sP_{t-s}=P_t$. From the last bound together with (\ref{neq0}) and (\ref{nnn1}) we finally get
\begin{align*}P_s \left(\frac{g(x)}{P_{t-s}f(x)}\right)\leq&(2+4(t-s)^2M^{2}C^2_{1.1})P_{t}\left(\frac{g(x)}{f(x)}\right) \\ &+  (t-s)^2 4 M^2 C^2_{1.1}  \sum_{ j = 1 }^N\  P_{t}\left(\frac{g(\Delta_j(x))}{ f(x)}\right).\end{align*} 
\end{proof}Before we present the proof of the  Theorem \ref{thmCon},  we show a sweeping out relationship for the carr\'e du champ.
\begin{lemma}\label{lastLem}For  the PJMP as described in  (\ref{eq:generator0})-(\ref{phi2}), we have \begin{align*}
\Gamma  (P_{t-s}f,P_{t-s}f)(x)  \leq & 
2\Gamma(f,f)( x) + 2c \sum_{i=1}^N \phi(x^{i})  \Gamma(f,f)(\Delta_i(x))+ \\  & +2Mc(t-s)  P_{t-s}\left(\frac{\Gamma (f , f)(x)}{f(x)}\right) P_{t-s}(f(x)),
 \end{align*}
 for   $c$ and $c(t)$ as defined in Corollary \ref{techn}. 
\end{lemma}
\begin{proof}
 From the definition of the carr\'e du champ
  \begin{align} 
\Gamma  (P_{t-s}f,P_{t-s}f)(x) = \sum_{ i = 1 }^N \phi (x^i) ( \E^{\Delta_i ( x) }  f(x_{t-s})-\E^{x} f(x_{t-s}))^2. 
\label{eq2} \end{align}
If we use the  Dynkin's formula 
$$\E^x f(x_t)=f(x)+\int_0^{t}\E^x(\Ge f(x_u))du$$we get
  \begin{align*}\nonumber
\left(\E^{\Delta_i ( x)}f(x_{t-s})-\E^{x}f(x_{t-s})\right)^2   \leq &2  \left(f(\Delta_i ( x) )- f(x)\right)^2 +\\  +2& \left(\int_0^{{t-s}}\left(\E^{\Delta_i ( x) }( \Ge (f(x_{u}))-  \E^x(\Ge (   f(x_{u}))\right)du\right)^2. 
 \end{align*}
  In order to bound the second term above we will use the bound shown in  Corollary  \ref{techn}  
     \begin{align*}\nonumber
\left(\E^{\Delta_i ( x)}f(x_{t-s})-\E^{x}f(x_{t-s})\right)^2   \leq & 2   \left(f(\Delta_i ( x) )- f(x)\right)^2 + 2c    \Gamma(f,f)(\Delta_i(x)) + \\  & +2c(t-s) P_{t-s}\left(\frac{\Gamma (f , f)(x)}{f(x)}\right) P_{t-s}(f(x)).
 \end{align*}
where $c$ and $c(t)$ as in Corollary \ref{techn}. This  together with   (\ref{eq2}) gives
  \begin{align*}
\Gamma  (P_{t-s}g,P_{t-s}g)(x)  \leq & 
2\Gamma(f,f)( x) + 2c \sum_{i=1}^N \phi(x^{i})  \Gamma(f,f)(\Delta_i(x))+ \\  & +2Mc(t-s)  P_{t-s}\left(\frac{\Gamma (f , f)(x)}{f(x)}\right) P_{t-s}(f(x)).
 \end{align*}

\end{proof}

We have obtained all the technical results that we need  to prove Theorem \ref{thmCon}. 
\subsection{proof of Theorem \ref{thmCon}:}\label{mainproof}

˜

 We will work similar to \cite{A-L}.  Denote $P_tf(x)=\E^xf(x_t)$. If we define $\phi(s)=P_s(P_{t-s}f \log P_{t-s}f))$ then, for every $f\geq 0$
  \begin{align*}
\phi '(s)=\frac{1}{2}P_s \left( \Ge (P_{t-s} f \log P_{t-s} f  )-(1+\log P_{t-s}f)\Ge (P_{t-s}f) \right)  \end{align*}
 where above we used that for a semigroup and its associated infinitesimal generator  the following well know relationships: $\frac{d}{ds}P_s= \Ge P_s =P_s \Ge$ (see for example \cite{G-Z}).

Since $\log a- \log b\leq \frac{(a-b)}{b}$ we have 
\[\Ge (f \log f)-(1+\log f)\Ge f \leq \frac{2 \Gamma(f,f)}{f}.\] Using this we get
    \begin{align*}
\phi '(s)\leq P_s \left( \frac{1}{P_{t-s}f}\Gamma(P_{t-s}f,P_{t-s}f) \right). \end{align*} 
 If we use  Lemma \ref{lastLem} to bound the carr\'e du champ of the semigroup we get 
  \begin{align*}
\phi '(s)\leq &  2 P_s \left( \frac{\Gamma(f,f)( x) }{P_{t-s}f(x)}   \right)+2c  M\sum_{i=1}^N  P_s \left( \frac{\Gamma(f,f)(\Delta_i(x))}{P_{t-s}f(x)}   \right)  +
\\  &  +2Mc(t) P_s  P_{t-s}\left(\frac{\Gamma (f , f)(x)}{f(x)}\right), \end{align*}  
since $c(t-s)\leq c(t)$.  We can use  Lemma \ref{neolemma} to bound the first and second term as well as the  the semigroup property  $P_s P_{t-s}=P_t$. We will then get
 \begin{align*}
\phi '(s)\leq &  \alpha'(t)P_{t}\left(\frac{\Gamma(f,f)(x)}{f(x)}\right) +  \beta'(t)\sum_{ j = 1 }^N P_{t}\left(\frac{ \Gamma(f,f)(\Delta_j(x))}{ f(x)}\right)+  \\  &+ \gamma'(t) \sum_{i=1}^N    \sum_{ j = 1 }^N P_{t}\left(\frac{ \Gamma(f,f)(\Delta_i(\Delta_j(x)))}{ f(x)}\right)
 \end{align*}  
 for $\alpha'(t)=  2Mc(t)+2d(t)$, $\beta'(t)=2(cM+1)d(t)  $ and $\gamma(t)=2cMd(t)$.
If we integrate, we will finally obtain 
 \begin{align*} \phi(t)-\phi(0)= &P_t( f\log f)-P_t f \log P_t f  \\  \leq & \alpha(t)P_{t}\left(\frac{\Gamma(f,f)(x)}{f(x)}\right) +  \beta(t)\sum_{ j = 1 }^N\ P_{t}\left(\frac{ \Gamma(f,f)(\Delta_j(x))}{ f(x)}\right)+  \\  &+ \gamma(t) \sum_{i=1}^N  \sum_{ j = 1 }^N P_{t}\left(\frac{ \Gamma(f,f)(\Delta_i(\Delta_j(x)))}{ f(x)}\right)
 \end{align*}  
where $\alpha(t)=t\alpha'(t), \beta(t)=t \beta'(t)$ and $\gamma(t)=t\gamma'(t)$. Then the proof is completed for $\delta(t)=\max\{\alpha(t),\beta(t),\gamma(t)\}$.
 
 \subsection{proof of Corollary \ref{corP}:}\label{proofcor}

˜

In order to prove the corollary, it is sufficient, to bound  the  carr\'e du champ operators   $\Gamma (f_k,f_k)(\Delta_i(x))$ and   $\Gamma (f_k,f_k)(\Delta_i(\Delta_j(x)))$,  by $\Gamma (f_k,f_k)(x)$, for all $k=1,...,N$. 

A)  Consider the case where $f_k$ is   decreasing  and convex. Then by the convexity  we have  $f_k(\frac{z+y}{2})\leq \frac{f_k(z)+f_k(y)}{2}$ for every $x,y$ on the domain of $f_k$. 

A1) At first assume that  $k\neq i$.   We compute
  \begin{align}\label{2.4_0}\Gamma (f_k,f_k)(\Delta_i(x))= & \frac{1}{2} \sum_{ j = 1 }^N \phi (\Delta_i ( x)^j) (f_k (\Delta_j ( \Delta_i ( x)))-f_k (\Delta_i ( x))^2  \\     \leq   &\frac{M}{\delta} \frac{1}{2} \sum_{ j = 1, j\neq k }^N \phi (x^j)  \left(f_k (x+2w)-f_k (x+w)\right)^{2}\\  &+\frac{M}{\delta} \frac{1}{2}  \phi (x^k)  \left(f_k (0)-f_k (x+w)\right)^{2}\nonumber
 \end{align}
 where above  we used (\ref{phi2}) and  $\phi \leq M$ and that   the weights are all equal $w_{ji}=w$.

To bound the first term on the  right hand side of (\ref{2.4_0}) we choose $z=x$ and $y=2x+w$. We then have 
 \[f_k(x+w)\leq \frac{f_k(x)+f_k( x+2w)}{2}\]\[\Updownarrow\]\[ \left(f_k(x+w)-  f_k( x+2w)\right)^{2}\leq \left(f_k(x)  -f_k(x+w) \right)^{2}\]
 since $f_k$ is decreasing and the weight $w$ is positive. This implies 
  \begin{align}\label{2.4_1} \frac{M}{\delta} \frac{1}{2} \sum_{ j = 1, j\neq k }^N \phi (x^j)  \left(f_k (x+2w)-f_k (x+w)\right)^{2}\leq\frac{M}{\delta} \Gamma(f_k,f_k)(x).\end{align}
 In order to bound the second term on the right hand side of (\ref{2.4_0}), we   choose $z=0$ and $y=x+w$, for $x>0$. Then, we have
 \begin{align}\label{x>w1}  f_k(\frac{x+w}{2})\leq \frac{f_k(x+w)+f_k(0)}{2} .  \end{align}
 But in the domain of the invariant measure $D$, the membrane potential of a neuron can be either  bigger and equal to $w$, or equal to zero. This is because when a neuron $k$ spikes, its membrane potential is set to $0$. Then, it can only leave from zero when another neuron spikes, in which case it goes to $w$. In this way, we have that either $x^{k}=0$ or $x^{k}\geq w$. As a result,   $x>0$, implies that 
 \[x\geq w \iff  \frac{x+w}{2}\leq x \iff  f_k(x)\leq f_k(\frac{x+w}{2}) \]
 since $f_k$ is decreasing. Then  (\ref{x>w1}) becomes
  \[  f_k(x)\leq \frac{f_k(x+w)+f_k(0)}{2}   \iff  f_k(0)-f_k(x+w)\leq 2( f_k(0)-  f_k(x))  \]
which, since $f_k$ is decreasing, readily implies
   \begin{align}\label{fk1}\frac{M}{\delta} \frac{1}{2}  \phi (x^k)  \left(f_k (0)-f_k (x+w)\right)^{2}\leq \frac{2M}{\delta} \frac{1}{2}  \phi (x^k) \left(   f_k(0)-  f_k(x)\right)^{2}  \end{align}
   for every $x>0$. In the   case where $x=0$, we observe that, for any $j\neq k$
   \[f_k(0)-f_k(w)=f_k(0)-f_k(\Delta_j( 0)),\]
   and so
   \[\frac{M}{\delta} \frac{1}{2}  \phi (0)  \left(f_k (0)-f_k (w)\right)^{2}\leq \frac{M}{\delta} \Gamma(f_k,f_k)(0).\]
   Combining this last bound together with (\ref{fk1}), we obtain 
  \begin{align}\label{fk2}\frac{M}{\delta} \frac{1}{2}  \phi (x^k)  \left(f_k (0)-f_k (x+w)\right)^{2}\leq  \frac{2M}{\delta} \Gamma(f_k,f_k)(x)\end{align}
   for every $x\in D$. Finally, if we use   (\ref{fk2}) to bound (\ref{2.4_0}), we get
   \begin{align}\label{VLE1} \Gamma (f_k,f_k)(\Delta_i(x)) \leq  \frac{2M}{\delta} \Gamma(f_k,f_k)(x).\end{align}
A2)    Now assume $k=i$. Since, for every $x>0$,  $\Delta_k(x^{k})=0$, we can compute
  \begin{align}\label{2.4_20}\Gamma (f_k,f_k)(\Delta_k(x))=  \frac{1}{2} \sum_{ j = 1 }^N \phi (0) (f_k (\Delta_j ( 0))-f_k (0))^2      \leq    \frac{N-1}{2}   \phi (0)  \left(f_k (w)-f_k (0)\right)^{2}.
 \end{align}
 But since $f_k$ is a decreasing function, one gets that for every $x>0$, $f_k(x)\leq f_k(w)$, since every $x\in D$, that is $x>0$, satisfies $x\geq w$. As a result we have
  \begin{align*}\left(f_k (0)-f_k (w)\right)^{2}\leq &\left(f_k (0)-f_k (x)\right)^{2}\\  =&\left(f_k (\Delta_k(x)-f_k (x)\right)^{2}\\  \leq &\frac{1}{\delta}\phi(x^{k})\left(f_k (\Delta_k(x)-f_k (x)\right)^{2}\\ \leq&  \frac{2}{\delta}\Gamma(f_k,f_k)(x).\end{align*}
Putting this in (\ref{2.4_20}) we get
\[\Gamma (f_k,f_k)(\Delta_k(x))\leq    \frac{(N-1) \phi (0)}{\delta}   \Gamma(f_k,f_k)(x), \] 
 for every $x>0$. In the case where $x=0$, the same result follows trivially. 
 
B)  Now, consider the case where $f_k$  is increasing and concave. Since this case follows similar to the decreasing and convex case studied in (A), we just highlight the differences. By the concavity we have  $f_k(\frac{z+y}{2})\geq \frac{f_k(z)+f_k(y)}{2}$ for every $x,y$ on the domain of $f_k$. To bound the first term on the  right hand side of (\ref{2.4_0}) if we  choose  $z=x$ and $y=2x+w$, since $f_k$ is increasing, we obtain
$ \left(  f_k( x+2w)-f_k(x+w)\right)^{2}\leq \left(f_k(x+w) -f_k(x)  \right)^{2},$ which 
 implies  (\ref{2.4_1}).
 
  To bound the second term on the right hand side of (\ref{2.4_0})   choose $z=0$ and $y=x+w$, for $x>0$. Then, 
 \[f_k(x)\geq f_k(\frac{x+w}{2})\geq \frac{f_k(x+w)+f_k(0)}{2} ,  \]
 which leads to 
  \[   f_k(x+w)-f_k(0)\leq 2(  f_k(x)- f_k(0))  \]
and so   (\ref{fk1}) 
   for every $x>0$ is again satisfied. The   case  $x=0$, follows identically to (A) and so (\ref{fk2}) is verified    for every $x\in D$, and so (\ref{VLE1}) follows for $k\neq i$.
       Now assume $k=i$. Since  for every $x>0$,  $\Delta_k(x^{k})=0$, we can compute
  \begin{align}\label{2.4_2}\Gamma (f_k,f_k)(\Delta_k(x))&=  \frac{1}{2} \sum_{ j = 1 }^N \phi (0) (f_k (\Delta_j ( 0))-f_k (0))^2  \\     & \leq    \frac{N-1}{2}   \phi (0)  \left(f_k (w)-f_k (0)\right)^{2}.\nonumber
 \end{align}
 But since $f_k$ is a decreasing function, one gets that for every $x>0$, $f_k(x)\leq f_k(w)$, since every strictly positive $x\in D$  satisfies $x\geq w$. As a result we have
  \begin{align*}\left(f_k (0)-f_k (w)\right)^{2}\leq &\left(f_k (0)-f_k (x)\right)^{2}\\ =&\left(f_k (\Delta_k(x)-f_k (x)\right)^{2}\\  \leq &\frac{1}{\delta}\phi(x^{k})\left(f_k (\Delta_k(x)-f_k (x)\right)^{2}\\  \leq & \frac{2}{\delta}\Gamma(f_k,f_k)(x).\end{align*}
Putting this in (\ref{2.4_2}) we get
\[\Gamma (f_k,f_k)(\Delta_k(x))\leq    \frac{(N-1) \phi (0)}{\delta}   \Gamma(f_k,f_k)(x), \] 
 for every $x>0$. In the case where $x=0$, the same result as in (A) follows trivially.

   Now assume  $k=i\neq j$.   Since, for every $x>0$,  $\Delta_k(x^{k})=0$, as before we have
  \begin{align}\label{2.4_2a}\Gamma (f_k,f_k)(\Delta_k(x)    \leq    \frac{N-1}{2}   \phi (0)  \left(f_k (w)-f_k (0)\right)^{2}.
 \end{align}
 But since $f_k$ is an increasing function, one gets that for every $x>0$, $f_k(x)\geq f_k(w)$, and so
  \begin{align*}\left(f_k (w)-f_k (0)\right)^{2}\leq &\left(f_k (x)-f_k (0)\right)^{2}=\left(f_k (x)-f_k (\Delta_k(x))\right)^{2}  \leq   \frac{2}{\delta}\Gamma(f_k,f_k)(x).\end{align*}
Putting this in (\ref{2.4_2}) we get
\[\Gamma (f_k,f_k)(\Delta_k(x))\leq    \frac{(N-1) \phi (0)}{\delta}   \Gamma(f_k,f_k)(x), \] 
 for every $x>0$. In the case where $x=0$, the same result follows trivially.  
 
 Combining together (A) and (B), we conclude that there exists a $R>0$ such that
 \begin{align*}\Gamma (f_k,f_k)(\Delta_i(x))\leq   R \Gamma(f_k,f_k)(x). \end{align*}
 Then, by induction we further get
\[\Gamma (f_k,f_k)(\Delta_i(\Delta_j(x)))\leq   R \Gamma(f_k,f_k)(\Delta_j(x)) \leq R^{2}  \Gamma(f_k,f_k)(x),\]
and the corollary  is proven.
\qed

\section{concentration and empirical approximation.}\label{secCon}
Concentration inequalities
\[\mu \left(\vert f- \mu (f) \vert >r \right)\leq c e^{-\lambda  r^p},   \   \   p\geq 1,\]
for a  probability measure $\mu$, have  been associated by  Talagrand  (see \cite{Ta1}  and \cite{Ta2})
to Poincar\'e and   log-Sobolev   inequalities (see also \cite{Bo-Le}, \cite{Led} and \cite{Led0}). In the current section we prove the typical concentration inequality of Proposition   \ref{Pcon} and the concentration empirical approximation of  Theorem  \ref{empthm}.  We begin with the proof of Theorem   \ref{empthm}.

\subsection{proof of the concentration empirical approximation of  Theorem  \ref{empthm}.}\label{proofEmpCon}

The concentration empirical approximations, as in  Theorem  \ref{empthm}, demonstrate how sharp are the approximations of the process as the number  of observables goes to infinity (see \cite{Ma} and \cite{B-G-V} for strong concentration  results). As briefly mentioned in the introduction,  these do not follow from the concentration inequalities of Proposition \ref{Pcon}, since, the empirical concentration inequalities involve non bounded quantities, as is the       sum $\sum_{k=1}^{n}f(X^{i}_{t_k})$ for an increasing to infinity $n$.  If one tries to  obtain Theorem \ref{empthm} from Proposition \ref{Pcon} by     applying the proposition to the bounded quantity $\frac{\sum_{k=1}^{n}f(X^{i}_{t_k})}{n}$,  will then obtain  
   \begin{align*}   P\left(\left\vert \frac{\sum_{k=1}^{n}f(X^{i}_{t_k})}{n}-\frac{\sum_{k=1}^{n}\E^x [f(X^{i}_{t_k})]}{n}\right\vert \geq  \epsilon  \right)\leq Q  e^{-\epsilon  } ,  \end{align*} 
   which is a concentration property whose right hand side does not depend on  $n$. In general, to obtain the concentration inequalities of Theorem  \ref{empthm}, one needs to first derive a modified log-Sobolev inequality for a set of cylindrical functions, independent of the size of input $n$. Such an inequality is obtained in (\ref{cylMS}). Then, from this inequality,  the concentration property for the  multi-times  function is obtained.

We start with some technical results.   
\begin{lemma} \label{lemma6}Assume $f_i:[0,m]\rightarrow [0,\infty)$ is a 1-Lipschitz function  depending only on $x^{i}$. Then, for any $x\in D$ and any $k\in \N$,  
\begin{align*}e^{\lambda f_i(\Delta_{j_1}(...\Delta_{j_k}(x))}\leq d ^{k}e^{\lambda f_i(x)}
 \end{align*}
 and 
 \begin{align*} \Gamma (e^{\lambda f_i},e^{\lambda f_i})(\Delta_{j}(x))\leq \lambda^{2}d e^{2\lambda f_i(x)},
 \end{align*}
 for any $\lambda\leq 1$ and $d=\frac{M (\max_ {ij}w_{ij})^{2}}{2}   e^{2 m}>1$.
\end{lemma}
\begin{proof} 
At first  we show the first assertion of the lemma. For this, it is sufficient to show
\[e^{\lambda f_i} (\Delta_j(x))\leq d  e^{\lambda f_i}(x) \]
for some constant $d$ independent of $i$. Then the result follows by induction. 
We can write 
\[e^{\lambda f_i} (\Delta_j(x))=e^{\lambda (f_i(\Delta_j(x^{i}))-f_i(x^{i}))}e^{\lambda  f_i(x^{i})}\leq e^{\lambda ( \Delta_j(x^{i})- x^{i}) }e^{\lambda  f_i(x^{i})}.\]
If $j=i$, then $\vert \Delta_j(x^{i})- x^{i}\vert=x^{i}\leq m$, while, if $j\neq i$,  $\vert \Delta_j(x^{i})- x^{i}\vert=w_{ji}\leq \max_{ji} w_{ji}$, and so the result follows.

Now we show the second assertion.  
 If $j=i$, then $\Delta_j(x^{i})=0$ and so 
\begin{align*}\Gamma(e^{\lambda f_i},e^{\lambda f_i})(\Delta_i(x))=&\frac{1}{2}\sum_{k = 1 }^N \phi (x^k) \left[ e^{\lambda f_i(w_{ji} )} - e^{\lambda f_i(0)}   \right]^2  \\ \leq &\frac{\lambda^{2}}{2}\sum_{k = 1 }^N \phi (x^k) \left[   f_i(w_{ji} )- f_i(0)  \right]^2 e^{2\lambda \max\{ f_i(0),f_i(w_{ji})\}}
.\end{align*}
 But $f_i$ is   Lipschitz continuous, with Lipschitz constant one, which means that  $\left[   f_i(w_{ji} )- f_i(0)  \right]^2\leq w_{ji}^{2}$, and that for every $y\in\{0,w_{ji}\}:$ $f_i(y)=f_i(y)-f_i(x)+f_i(x)\leq \vert y-x\vert+f_i(x)\leq m+ f_i(x)$, since $x\leq m$. So,
 \begin{align*}\Gamma(e^{\lambda f_i},e^{\lambda f_i})(\Delta_i(x)) \leq \lambda^{2}\frac{M (\max_ {ij}w_{ij})^{2}}{2}   e^{2\lambda m} e^{2\lambda  f_i(x)}. \end{align*}
If $j\neq i$, then  $\Delta_j(x^{i})=x^{i}+w_{ji}$ and so 
\begin{align*}\Gamma(e^{\lambda f_i},&e^{\lambda f_i})(\Delta_j(x))=\frac{1}{2}\sum_{k = 1 }^N \phi (x^k) \left[ e^{\lambda f_i(x^{i}+w_{ji}+w_{ki})} - e^{\lambda f_i(x^{i}+w_{ji})}  \right]^2\leq 
 \\  &   \leq\frac{\lambda^{2}}{2}\sum_{k = 1 }^N \phi (x^k) \left[   f_i(x^{i}+w_{ji}+w_{ki} )- f_i(x^{i}+w_{ji})  \right]^2 e^{2\lambda \max\{  f_i(x^{i}+w_{ji}+w_{ki} ), f_i(x^{i}+w_{ji})\}} .
\end{align*}
But $\max\{  f_i(x^{i}+w_{ji}+w_{ki} ), f_i(x^{i}+w_{ji})\}\leq \vert w_{ji}+w_{ki}\vert+f_i(x)$, which leads to
\begin{align*}\Gamma(e^{\lambda f_i},e^{\lambda f_i})(\Delta_j(x))  \leq \lambda^{2}\frac{M (\max_ {ij}w_{ij})^{2}}{2}   e^{2\lambda  \max_ {ij}w_{ij}} e^{2\lambda  f_i(x)} .
\end{align*}

\end{proof}
Since in this section we will be concerned with cylindrical functions, in order to ease the notation, for any sequence   of times $0=t_1<t_2<...<t_n\leq  T$ and a  function $F:\R^n\rightarrow\R,$ depending on $\{X_{t_k},k=1,...,n\},$  we  will  write $\hat P_{t_n}F(x):=\E^x[F(X_{t_1},...,X_{t_n})]  $. When $n=1$,  this is of course the Markov semigroup, $\hat P_{t_1}F(x)=P_{t_1}F(x)=\E^x[F(X_{t_1})]$. 

The following lemma will be the main iteration tool that will be used in order to obtain the modified log-Sobolev inequality for cylindrical functions shown in  (\ref{cylMS}).
\begin{lemma} \label{InducLem} Assume  the PJMP as described in  (\ref{eq:generator0})-(\ref{phi2}). Assume  $\xi:D^{n}\rightarrow D^{n}$ and $f:\R^{n} \rightarrow \R$ are  such that $\frac{f(\xi(y))}{f(y)}\leq d^{2}$ for any $y\in D$, for some $d>1$. Then, for any sequence of times $0=t_1<t_2<...<t_n\leq  T$, we have
 \begin{align*}
 P_{t_{k+1}-t_{k}}&\left(\frac{ \Gamma  (\hat P_{t_{k}}f, \hat P_{t_{k}}f)(\xi(x))}{ \hat P_{t_{k}}f(x)} \right) \leq\\ & b_k\sum_{r=1}^{k+2}\left[ \sum_{ i_r = 1 }^N...\sum_{ i_2 = 1 }^N \sum_{i_1=1}^N \hat P_{t_{k+1}}\left(\frac{  \Gamma(f,f)(\Delta_{i_{r+2}}(...\Delta_{i_2}( \Delta_{i_1}(\xi(x))
))}{f(x)} \right)\right],
 \end{align*}
 where above we have denoted  $b_k=3^{k}b_0(k),$ for   $b_0(k)=\max\{2+2d^{2}Mc(t_{k}-t_{k-1}),2c M+2 d(t_k-t_{k-1}), 2cM d(t_k-t_{k-1})\}$. 
\end{lemma}
\begin{proof}  We can write
 \begin{align*} 
\Gamma  (\hat P_{t_{k}}f, \hat P_{t_{k}}f)(\xi(x))= \Gamma  ( P_{t_{k}-t_{k-1}} \hat P_{t_{k-1}}f, P_{t_{k}-t_{k-1}} \hat P_{t_{k-1}}f)(\xi(x)),
 \end{align*}
where above $ \hat P_{t_{k}}$ and  $ \hat P_{t_{k-1}}$ refer to the expectation with respect to $\{X_{t_1},...,X_{t_k}\}$ and $\{X_{t_1},...,X_{t_{k-1}}\}$ respectively. One should notice that $ P_{t_k-t_{k-1}}$ is the semigroup related to   $X_{t_{k}}$ (which appears after time    $t_{k}-t_{k-1}$  from $X_{t_{k-1}}$) and so we can apply the tools we have already obtained about the Markov semigroups in previous sections.  We can start by   applying the sweeping out relationship of Lemma \ref{lastLem} to get the semigroup out of the carr\'e du champ 
  \begin{align*}
\Gamma  (\hat  P_{t_{k}}f, & \hat P_{t_{k}}f)(\xi(x)) \leq 
2\Gamma(\hat  P_{t_{k-1}}f,  \hat P_{t_{k-1}}f)( \xi(x))+\\  & + 2c \sum_{i=1}^N \phi(\xi(x)^{i})  \Gamma(\hat  P_{t_{k-1}}f,  \hat P_{t_{k-1}}f)(\Delta_i(\xi(x)))+ \\  & +2Mc(t_{k}-t_{k-1})   P_{t_{k}-t_{k-1}}\left(\frac{\Gamma (\hat  P_{t_{k-1}}f,  \hat P_{t_{k-1}}f)(\xi(x))}{  \hat P_{t_{k-1}}f(\xi(x))}\right)   \hat P_{t_{k}}(f(\xi(x))),
 \end{align*}
 since $  P_{t_{k}-t_{k-1}}  \hat P_{t_{k-1}}=  \hat P_{t_{k}}$.
 From this we obtain
  \begin{align*}
  P_{t_{k+1}-t_{k}}&\left(\frac{ \Gamma  (\hat  P_{t_{k}}f,\hat  P_{t_{k}}f)(\xi(x))}{  \hat P_{t_{k}}f(x)} \right) \leq 
2  P_{t_{k+1}-t_{k}}\left(\frac{\Gamma(\hat  P_{t_{k-1}}f,  \hat P_{t_{k-1}}f)( \xi(x)) }{  \hat P_{t_{k}}f(x)} \right)+\\  &
+ 2cM \sum_{i=1}^N  P_{t_{k+1}-t_{k}}\left(\frac{ \Gamma(\hat  P_{t_{k-1}}f,  \hat P_{t_{k-1}}f)(\Delta_i(\xi(x)))}{  \hat P_{t_{k}}f(x)} \right)+ \\  & +2d^{2}Mc(t_{k}-t_{k-1})  P_{t_{k+1}-t_{k-1}}\left(\frac{\Gamma (  \hat P_{t_{k-1}}f,  \hat P_{t_{k-1}}f)(\xi(x))}{  \hat P_{t_{k-1}}f(x)}\right) ,
 \end{align*}
 where above we also used  that  $\frac{  \hat P_{t_k}(f(\xi(x)))}{  \hat P_{t_k}(f(x))}\leq d^{2}$. This  follows from the hypothesis $f(\xi(y))\leq d^{2}f(y)$, $\forall y\in D,$   after taking the expectation on both sides with respect to the   $  \hat P_{t_k}$, that is,  $   \hat P_{t_k}(f(\xi(x)))\leq d^{2}  \hat P_{t_k}(f(x))$. 
 We can write the denominator of the first two terms on the right hand side as $  \hat P_{t_k}=  P_{t_k-t_{k-1}}  \hat P_{t_{k-1}}$, and then apply Lemma  \ref{neolemma}, to reduce the  denominator from  $\hat P_{t_k}$ to $\hat P_{t_{k-1}}$, as shown below
  \begin{align*}
 P_{t_{k+1}-t_{k}}&\left(\frac{ \Gamma  (\ \hat P_{t_{k}}f, \hat P_{t_{k}}f)(\xi(x))}{ \hat P_{t_{k}}f(x)} \right) \leq \\  & 
(2+2d^{2}Mc(t_{k}-t_{k-1}) ) P_{t_{k+1}-t_{k-1}}\left(\frac{\Gamma(\hat P_{t_{k-1}}f, \hat P_{t_{k-1}}f)(\xi(x)) }{ \hat P_{t_{k-1}}f(x)} \right)+\\  +&
( 2c M+2 d(t_k-t_{k-1}))\sum_{i_1=1}^N P_{t_{k+1}-t_{k-1}}\left(\frac{   \Gamma(\hat P_{t_{k-1}}f, \hat P_{t_{k-1}}f)(\Delta_{i_1}(\xi(x)))}{ \hat P_{t_{k-1}}f(x)} \right)+ \\ + &
 2cM d(t_k-t_{k-1})\sum_{ i_2 = 1 }^N \sum_{i_1=1}^N P_{t_{k+1}-t_{k-1}}\left(\frac{  \Gamma(\hat P_{t_{k-1}}f, \hat P_{t_{k-1}}f)(\Delta_{i_2}( \Delta_{i_1}(\xi(x))
))}{ \hat P_{t_{k-1}}f(x)} \right).
 \end{align*}
 If  we define $b_0(k)=\max\{2+2d^{2}Mc(t_{k}-t_{k-1}),2c M+2 d(t_k-t_{k-1}), 2cM d(t_k-t_{k-1})\}$, then by induction  we finally get
  \begin{align*}
 P_{t_{k+1}-t_{k}}&\left(\frac{ \Gamma  (\hat P_{t_{k}}f, \hat P_{t_{k}}f)(\xi(x))}{ \hat P_{t_{k}}f(x)} \right) \leq \\ & b_k\sum_{r=1}^{k+2}\left[ \sum_{ i_r = 1 }^N...\sum_{ i_2 = 1 }^N \sum_{i_1=1}^N  \hat P_{t_{k+1}}\left(\frac{  \Gamma(f,f)(\Delta_{i_{r+2}}(...\Delta_{i_2}( \Delta_{i_1}(\xi(x))
))}{f(x)} \right)\right],
 \end{align*}
 for  $b_k=3^{k}b_0(k)$, where above we used that $t_0=0$ and $ \hat P_0f(x)=P_{0}f(x)=f(x)$.
  
\end{proof}
In the next proposition   a bound of the entropy  of  multi-times  functions is presented. Furthermore,  in the process of proving this bound   a modified log-Sobolev inequality is also established, in (\ref{cylMS}), for  multi-times  functions.  
\begin{proposition}\label{Propo2}Consider some $T>0$ and     any sequence of times $0=t_1<t_2<...<t_n\leq  T$,  such that \[D(T)=3 \sum_{k=1}^{\infty}  \delta(t_k-t_{k-1})   b_{k-1}\sum_{r=1}^{k+1} (Nd)^{r+4}<\infty.\] 
 For $i=1,...,N$, assume $f_i:[0,m]\rightarrow [0,\infty)$ is a 1-Lipschitz function  depending only on $x^{i}$ and $f(X^{i}_{t_1},...,X^{i}_{t_n}):=\sum_{k=1}^{n}f_i(X^{i}_{t_k})$.  
Then, for $\lambda\leq 1$,
 \begin{align*}   \hat P_{t_{n}}( e^{\lambda f(X^{i}_{t_1},...,X^{i}_{t_n})}\log \frac{e^{\lambda f(X^{i}_{t_1},...,X^{i}_{t_n})}}{ \hat P_{t_{n}} e^{\lambda f(X^{i}_{t_1},...,X^{i}_{t_n})}})\leq 
 \lambda^{2} D(T)\hat
    P_{t_{n}}\left(e^{\lambda f(X^{i}_{t_1},...,X^{i}_{t_n})} \right).
 \end{align*}  
 \end{proposition}
 \begin{proof} To prove the  multi-times  modified  log-Sobolev inequality of the proposition, we will take advantage of the modified inequality shown on Theorem \ref{thmCon} and then use iteration, as was done in \cite{W-Y} and \cite{Chaf} to prove coersive inequalities for cylindrical functions.  To do so, we will first form the entropies for the successive times. To ease the notation,  we will write $f(x)$ for $f(X^{i}_{t_1},...,X^{i}_{t_n})$. We have
 \begin{align}\label{inducEn}   \hat P_{t_{n}}(e^{\lambda f(x)}\log \frac{e^{\lambda f(x)}}{ \hat P_{t_{n}} e^{\lambda f(x)}})= 
 \sum_{k=1}^{n} P_{t_{n}-t_k} P_{t_k-t_{k-1}}(\hat  P_{t_{k-1}}e^{\lambda f(x)} \log \frac{  \hat P_{t_{k-1}}e^{\lambda f(x)}}{ P_{t_k-t_{k-1}} \hat P_{t_{k-1}} e^{\lambda f(x)}}),
 \end{align} 
 where above we used that $ P_t \hat P_s= \hat P_{t+s}$.  We can now apply the modified log-Sobolev inequality of  Theorem \ref{thmCon} to bound the entropies of $ \hat P_{t_{k-1}}f$ with respect to the measure $ P_{t_k-t_{k-1}}$ involved in the sum. This   gives
   \begin{align*} P_{t_k-t_{k-1}}(\hat  P_{t_{k-1}}e^{\lambda f(x)} \log &\frac{  \hat P_{t_{k-1}}e^{\lambda f(x)}}{ P_{t_k-t_{k-1}} \hat P_{t_{k-1}} e^{\lambda f(x)}})\leq \\   &
  \delta(t_k-t_{k-1})  P_{t_k-t_{k-1}}\left(\frac{\Gamma (\hat P_{t_{k-1}}e^{\lambda f}, \hat P_{t_{k-1}}e^{\lambda f})(x)}{ \hat P_{t_{k-1}}e^{\lambda f}(x)}\right) \\ +&  \delta(t_k-t_{k-1})\sum_{j=1}^N   P_{t_k-t_{k-1}}\left(\frac{ \Gamma(\hat P_{t_{k-1}}f, \hat P_{t_{k-1}}f)(\Delta_j(x))}{  \hat P_{t_{k-1}}f(x)}\right)+  \\  +& \delta(t_k-t_{k-1}) \sum_{i,j=1}^N   P_{t_k-t_{k-1}}\left(\frac{ \Gamma ( \hat P_{t_{k-1}}f, \hat P_{t_{k-1}}f)(\Delta_i(\Delta_j(x)))}{ \hat P_{t_{k-1}} f(x)}\right).
 \end{align*}  
 If we use the last inequality to bound  the right hand side of (\ref{inducEn}) we get
  \begin{align*}  \hat P_{t_{n}}(e^{\lambda f(x)}&\log  \frac{e^{\lambda f(x)}}{ \hat P_{t_{n}} e^{\lambda f(x)}})\leq  \\  &
 \sum_{k=1}^{n} \delta(t_k-t_{k-1})    P_{t_{n}-t_{k}}P_{t_{k}-t_{k-1}}\left(\frac{\Gamma (\hat P_{t_{k-1}}e^{\lambda f}, \hat P_{t_{k-1}}e^{\lambda f})(x)}{ \hat P_{t_{k-1}}e^{\lambda f}(x)}\right) \\ +&  \sum_{k=1}^{n}  \delta(t_k-t_{k-1})\sum_{j=1}^N       P_{t_{n}-t_{k}}P_{t_{k}-t_{k-1}}\left(\frac{ \Gamma (\hat P_{t_{k-1}}f, \hat P_{t_{k-1}}f)(\Delta_j(x))}{  \hat P_{t_{k-1}}f(x)}\right)+  \\  +& \sum_{k=1}^{n}  \delta(t_k-t_{k-1}) \sum_{i,j=1}^N       P_{t_{n}-t_{k}}P_{t_{k}-t_{k-1}}\left(\frac{ \Gamma (\hat P_{t_{k-1}}f, \hat P_{t_{k-1}}f)(\Delta_i(\Delta_j(x)))}{ \hat P_{t_{k-1}} f(x)}\right).
 \end{align*}  
  To bound the righthand side we can use Lemma \ref{InducLem}
for  $P_{t_{k}-t_{k-1}}$, since   Lemma \ref{lemma6} guaranties that the main condition of Lemma \ref{InducLem}
 is satisfied. We then  get \begin{align}\label{cylMS} & \hat P_{t_{n}}(e^{\lambda f(x)}\log \frac{e^{\lambda f(x)}}{ \hat P_{t_{n}} e^{\lambda f(x)}})\leq \\ & \nonumber
 \sum_{k=1}^{n} \delta(t_k-t_{k-1})b_{k-1}\sum_{r=1}^{k+1}\left[ \sum_{ i_{r+2} = 1 }^N... \sum_{i_1=1}^N \hat P_{t_{n}}\left(\frac{  \Gamma(e^{\lambda f},e^{\lambda f})(\Delta_{i_{r+2}}(...\Delta_{i_2}( \Delta_{i_1}(x)
))}{e^{\lambda f(x)}} \right)\right]\\ +& \nonumber \sum_{k=1}^{n}  \delta(t_k-t_{k-1})\sum_{j=1}^N b_{k-1}\sum_{r=1}^{k+1}\left[ \sum_{ i_{r+2} = 1 }^N...\sum_{i_1=1}^N \hat P_{t_{n}}\left(\frac{  \Gamma(e^{\lambda f},e^{\lambda f})(\Delta_{i_{r+2}}(...( \Delta_{i_1}(\Delta_j(x))
))}{e^{\lambda f(x)}} \right)\right]  \\  +& \nonumber\sum_{k=1}^{n}  \delta(t_k-t_{k-1}) \sum_{i,j=1}^N b_{k-1}\sum_{r=1}^{k+1}\left[ \sum_{ i_{r+2} = 1 }^N... \sum_{i_1=1}^N \hat P_{t_{n}}\left(\frac{  \Gamma(e^{\lambda f},e^{\lambda f})(\Delta_{i_{r+2}}(...( \Delta_{i_1}(\Delta_i(\Delta_j(x)))
))}{e^{\lambda f(x)}} \right)\right],
 \end{align}  
 where above we used that $P_{t_{n}-t_{k}}P_{t_{k}-t_{k-1}}\hat P_{t_{k-1}}=\hat P_{t_n}$.  From Lemma \ref{lemma6}, we can further bound  
 \begin{align*} \Gamma (e^{\lambda f},e^{\lambda f})(\Delta_{j_r}(\Delta_{j_{r-1}}...\Delta_{j_1}(x)))\leq \lambda^{2}d e^{2\lambda f(\Delta_{j_{r-1}}...\Delta_{j_1}(x)))}\leq \lambda^{2}d^{r} e^{2\lambda f_i(x)} ,
 \end{align*}
where at first inequality we used the second assertion of the lemma to bound the carr\'e du champ and then the first assertion. From this we finally obtain  \begin{align*}  \hat P_{t_{n}}(e^{\lambda f(x)}\log \frac{e^{\lambda f(x)}}{ \hat P_{t_{n}} e^{\lambda f(x)}})\leq 3 \lambda^{2}\sum_{k=1}^{n}  \delta(t_k-t_{k-1})   b_{k-1}\sum_{r=1}^{k+1} (Nd)^{r+4} \hat P_{t_{n}}\left(e^{\lambda f(x)} \right),
 \end{align*}  
and so the proposition follows for $D(T)=3 \sum_{k=1}^{\infty}  \delta(t_k-t_{k-1})   b_{k-1}\sum_{r=1}^{k+1} (Nd)^{r+4}$.
  
   \end{proof}

\underline{proof of Theorem  \ref{empthm}: }

˜

We can now prove Theorem  \ref{empthm}.
Consider  $f(X^{i}_{t_1},...,X^{i}_{t_n})=\sum_{k=1}^nf(X^{i}_{t_k})$. For economy we will write  $f(x)$ for $f(X^{i}_{t_1},...,X^{i}_{t_n})$ and $X^{i}_{t_1}=x$. Then, for $\lambda \leq 1$, Proposition \ref{Propo2} gives 
 \begin{align*}   \hat P_{t_{n}}(e^{\lambda f(x)}\log \frac{e^{\lambda f(x)}}{ \hat P_{t_{n}} e^{\lambda f(x)}})\leq
 D(T)\hat
    P_{t_{n}}\left(e^{\lambda f(x)}\right).
 \end{align*}  
  Denote $\Psi(\lambda)=  \hat P_{t_{n}}\left( e^{ \lambda f}(x) \right) $. Then,  the last inequality can be written as 
 \[\frac{\lambda \Psi '(\lambda)}{\Psi (\lambda)}-\log \Psi (\lambda)\leq  \lambda^{2} D (T)e^{a \lambda }.\]
If we  now divide with $\lambda^{2}$ we then get 
 \[\frac{d}{d \lambda}\left(  \frac{\log \Psi (\lambda)}{\lambda} \right)\leq D(T)e^{a \lambda}.\]
 Since $\lim_{\lambda \rightarrow 0} \frac{\log \Psi (\lambda)}{\lambda }=\hat P_{t_{n}}(f(x))$, by integration we obtain
 \[ \hat P_{t_{n}}\left( e^{\lambda (f(x)-P_{t_{n}} f(x))}  \right)\leq e^{D(T)\lambda \int_0^{\lambda}e^{as}ds}\]
 and so
\[ \hat P_{t_{n}}\left( f(x)- \hat P_{t_n} f(x) \geq \epsilon  \right)\leq e^{-\epsilon \lambda}   \hat P_{t_{n}}\left(   e^{\lambda (f(x)-\hat P_{t_{n}}f(x))} \right)\leq  e^{-\epsilon \lambda} e^{D(T)\lambda \int_0^{\lambda}e^{as}ds}.\] 
 As a result,  if we consider $\lambda=1$ we have 
  \begin{align*}   \hat P_{t_{n}}\left(\frac{\sum_{k=1}^{n}f(X^{i}_{t_k})}{n}- \hat P_{t_{n}}\frac{\sum_{k=1}^{n}f(X^{i}_{t_k})}{n} \geq  \epsilon  \right)=&   \hat P_{t_{n}}\left(\sum_{k=1}^{n}X^{i}_{t_k}- \hat P_{t_{n}}\sum_{k=1}^{n}X^{i}_{t_k} \geq    n\epsilon  \right) \\  \leq &e^{-\epsilon  } \hat P_{t_{n}}\left(   e^{ (f(x)-\hat P_{t_{n}}f(x))} \right)\leq  e^{-\epsilon n} G   \end{align*} 
 where $G=e^{D(T) \int_0^{1}e^{as}ds}$. Next, if we repeat  the same for $-f$ in the place of $f$, we get 
   \begin{align*}   \hat P_{t_{n}}\left( \hat P_{t_{n}}\frac{\sum_{k=1}^{n}f(X^{i}_{t_k})}{n}-\frac{\sum_{k=1}^{n}f(X^{i}_{t_k})}{n} \geq  \epsilon  \right)\leq  e^{-\epsilon n} G.  \end{align*} 
 To finish the proof, it is insufficient to observe that      $D(T)<\infty$ implies  that $\delta(t_k-t_{k-1})   $ is uniformly bounded for any $k$, which implies the same for $t_k-t_{k-1}$. From this  we   conclude that $b_0(k)$ is also uniformly bounded on $k$. As a result, for  $D(T)<\infty,$ it is sufficient to assume   condition (\ref{condit1}) of the statement of the theorem.
  \qed

\subsection{proof of Proposition \ref{Pcon}:}\label{Pcon2}

 Consider $f(x)=\sum_{i=1}^{N}f_i(x^{i})$, for some $x=(x^{1},...,x^{N})$ and $0<\lambda\leq 1$. 
Then, from Theorem \ref{thmCon} we directly obtain  
\begin{align} \nonumber P_t&( e^{\lambda f(x)}\log \frac{e^{\lambda f(x)}}{P_t e^{\lambda f(x)}}) \leq   \delta(t)P_{t}\left(\frac{\Gamma(e^{\lambda f},e^{\lambda f})(x)}{e^{\lambda f}(x)}\right)+\\ & +  \delta(t)\sum_{j=1}^N  P_{t}\left(\frac{ \Gamma(e^{\lambda f},e^{\lambda f})(\Delta_j(x))}{ e^{\lambda f}(x)}\right)  + \delta(t) \sum_{i,j=1}^N    P_{t}\left(\frac{ \Gamma(e^{\lambda f},e^{\lambda f})(\Delta_i(\Delta_j(x)))}{ e^{\lambda f}(x)}\right).
\label{Pr1_0}\end{align}  
For the carr\'e du champ on the first term on the right hand side, we can compute
  \begin{align}\nonumber \Gamma (f,f)(x)= & \frac{1}{2} \sum_{ j = 1 }^N \phi (x^j) (e^{\lambda f} (\Delta_j (x))-e^{\lambda f} (x))^2  \\ \leq &\lambda^{2} 2^{N-1}  \sum_{ j = 1 }^N \phi (x^j) \sum_{i=1}^{N} (f_i (\Delta_j (x))-f _i(x))^2 e^{2\lambda \sum_{i=1}^{N}\max\{ f_i(x),f_i(\Delta_j (x))\}} . \nonumber \end{align}
 But, since $f_i$ is Lipschitz continuous,   with Lipschitz constant $1$, we can bound $\vert f_i (\Delta_j (x))-f _i(x)\vert\leq m+\max_{i,j}\{w_{ij}\}$ and $ f_i (\Delta_j (x))\leq  f_i (x)+m+\max_{i,j}\{w_{ij}\} $, and so  the last can be bounded as follows  \begin{align}  \label{Pr1_1}\Gamma (f,f)(x)\leq  &
   \lambda^{2}D_1e^{2\lambda f(x) } ,
 \end{align}
 where $D_1= 2^{N-1}MN (m+\max_{i,j}\{w_{ij}\})^2 e^{2  N(m+\max_{i,j}\{w_{ij}\})}$.
Similarly, for the second term of (\ref{Pr1_0}) we compute
  \begin{align}\label{Pr1_2}\Gamma (f,f)(\Delta_i(x))\leq  & \frac{1}{2} \sum_{ j = 1 }^N \phi (\Delta_i ( x)^j) (e^{\lambda f} (\Delta_j (\Delta_i ( x)))-e^{\lambda f} (\Delta_i ( x)))^2
 \\    \nonumber \leq &\lambda^{2} D_2 e^{2\lambda f(x) } \nonumber ,
 \end{align}
 for $D_2=2^{N-1} M  (m+\max_{i,j}\{w_{ij}\} )^2 e^{2  N( m+2\max_{i,j}\{w_{ij}\})} $, 
 where above we used that  $\vert f_i (\Delta_j (\Delta_i(x)))-f _i(\Delta_i(x))\vert\leq m+\max_{i,j}\{w_{ij}\}$ and $\max\{ f_k (\Delta_j (\Delta_i(x)))\}\leq  f_i (x)+m+2\max_{i,j}\{w_{ij}\} $. If we work as in the first two terms, the third term on the right hand side of (\ref{Pr1_0})  can be bounded by
  \begin{align}\label{Pr1_3}\Gamma (f,f)(\Delta_j(\Delta_i(x)))\leq  \lambda^{2} D_3 e^{2\lambda f(x) } ,
 \end{align}
 where now $D_3= 2^{N-1} M  (m+2\max_{i,j}\{w_{ij}\} )^2 e^{2  N( m+3\max_{i,j}\{w_{ij}\})}$. Combining together the bounds (\ref{Pr1_1})-(\ref{Pr1_3}) to bound the right hand side of (\ref{Pr1_0}) leads to 
\[  P_t( e^{\lambda f(x)}\log \frac{e^{\lambda f(x)}}{P_t e^{\lambda f(x)}}) \leq \lambda^{2}   \delta(t)\left( D_1+ND_2+N^{2}D_3 \right)P_{t}\left(e^{\lambda f(x)}\right) .
\]
 The rest of the  proof of Proposition \ref{Pcon}   follows on the same lines of the proof of Theorem \ref{empthm},  for  $\Psi(\lambda)=P_{t}\left(e^{\lambda f(x)}\right) $.
 \qed


\begin{thebibliography}{99}
   
 



 \bibitem{An} M. Andr\'e \textit{A   result   of   metastability   for   an   infinite   system   of   spiking   neurons}. J Stat Phys, 177, 984-1008 (2019)


 
 
\bibitem{A-L} C. Ane and M. Ledoux, \textit{Rate of convergence for ergodic continuous Markov processes: Lyapunov versus Poincar\'e}. Probab. Theory Relat. Fields 116, 573-602 (2000).
 
\bibitem{ABGKZ} R. Aza\"is, J.B. Bardet, A. Genadot, N. Krell and P.A. Zitt,
\textit{Piecewise deterministic Markov process (pdmps). Recent results}. Proceedings 44, 276-290 (2014). 



 \bibitem{Ba1} D. Bakry, \textit{L'hypercontructivit\'e et son utilisation en th\'eorie des semigroupes}. Ecole d'Et\'e de Probabilit\'es de St-Flour. Lecture Notes in Math., 1581, 1-114, Springer  (1994).  
 
  \bibitem{Ba2} D. Bakry, \textit{On  Sobolev and logarithmic Sobolev inequalities for Markov semigroups}. New trends in Stochastic Analysis, 43-75, World Scientific  (1997).  
 
 



 \bibitem{Bo-Le} S. Bobkov and M. Ledoux, \textit{Poincar\'e's inequalities and Talagrand's concentration phenomenon for the exponential distribution.} Probab Theory Relat Fields 107, 383-400 (1997).


  \bibitem{B-G-V} F. Bolley, A. Guillin and C. Villani, \textit{Quantitative Concentration Inequalities for Empirical Measures on Non-compact Spaces}.  C. Probab. Theory Relat. Fields 137, 541–593  (2007).
 
\bibitem{Chaf} D. Chafai,
 \textit{Entropies, convexity, and functional inequalities}. J. Math. Kyoto Univ. 44 (2), 325 - 363 (2004). 
 

 \bibitem{C17} J. Chevalier,
 \textit{Mean-field limit of  generalized Hawkes processes}.  Stochastic Processes and their Applications 127 (12), 3870 - 3912 (2017).

\bibitem{C-D-M-R}  A. Crudu, A. Debussche, A. Muller and O. Radulescu
\textit{Convergence of stochastic gene networks to hybrid piecewise deterministic processes}. The Annals of Applied Probability 22, 1822-1859, (2012).

\bibitem{Davis84} M.H.A. Davis
\textit{Piecewise-derministic Markov processes: a general class off nondiffusion stochastic models} J. Roy. Statist. Soc. Ser. B, 46(3) 353 - 388 (1984).

\bibitem{Davis93} M.H.A. Davis
\textit{Markov models and optimization} Monographs on Statistics and Applied Probability, vol. 49 Chapman $\&$ Hall,  London. (1993)

 
  



\bibitem{D-SC}  P. Diaconis and L. Saloff-Coste
\textit{Logarithmic   Sobolev inequalities for finite Markov Chanis}. The Annals of Applied Probability 6, 695-750, (1996).


 \bibitem{D-L-O} A. Duarte, E. L\"ocherbach and G. Ost, \textit{Stability, convergence to equilibrium and simulation of non-linear Hawkes Processes with memory kernels given by the sum of Erlang kernels }. ESAIM: PS 23,  770�796 (2019). 



\bibitem{D-O} A. Duarte and G.Ost,
\textit{A model for neural activity in the absence of external stimuli}  Markov Processes and Related Fields 22, 37-52  (2016). 
 
\bibitem{G-L} A. Galves and E. L\"ocherbach, \textit{Infinite Systems of Interacting Chains with Memory of Variable Length-A Stochastic Model for Biological Neural Nets}. J Stat Phys 151, 896-921 (2013). 
 


  \bibitem{G-Z} A.Guionnet  and B.Zegarlinski,
\textit{ Lectures on Logarithmic Sobolev Inequalities}, IHP Course 98,   1-134 in Seminare de Probabilite XXVI, Lecture Notes in
         Mathematics 1801, Springer (2003).
 
 \bibitem{H-R-R} N. Hansen, P. Reynaud-Bouret and V. Rivoirard
\textit{Lasso and probabilistic inequalities for multivariate point processes}. Bernoulli, 21(1) 83-143 (2015). 
 
\bibitem{H-K-L} P. Hodara, N. Krell and E. L\"ocherbach, \textit{Non-parametric estimation of the spiking rate in systems of interacting neurons}. E. Stat Inference Stoch Process,  1-16 (2016). 

\bibitem{H-L} P. Hodara and E. L\"ocherbach,
\textit{Hawkes processes with variable length memory and an infinite number of components}. Adv. Appl. Probab 49, 84-107 (2017).  

\bibitem{H-P} P. Hodara and I. Papageorgiou,
\textit{Poincar\'e type inequalities for compact  degenerate pure jump Markov processes}. Mathematics 7(6), 518 (2019).  

  
 \bibitem{Led} 
 M. Ledoux, \textit{The concentration of measure phenomenon.}  Mathematical Surveys and monographs 89, AMS  (2001). 
 
\bibitem{Led0} 
 M. Ledoux, \textit{Concentration of measure and logarithmic Sobolev inequalities.}  Seminaire de Probabilites XXXV. Lecture notes in Math. 1709, 120-216, Springer    (1999). 

 
\bibitem{L17} E. L\"ocherbach,
\textit{Absolute continuity of the invariant measure in piecewise deterministic Markov Processes having degenerate jumps}. Stoch. Process. Their Appl.  128, 1797-1829 (2018).



\bibitem{Ma} F. Malrieu,  \textit{Logarithmic Sobolev inequalities for some nonlinear PDE’s}. Stochastic Process. Appl. 95,  109-132 (2001).
 
\bibitem{No} J. R. Norris,  \textit{Markov Chains}. Vol 2  Cambridge Series in Statistical and Probabilistic Mathematics, Cambridge University Press, (1998). 
 
\bibitem{PTW-10} K. Pakdaman, M. Thieulen and G. Wainrib,
\textit{Fluid limit theorems for stochastic hybrid systems with application to neuron models}. Adv.\ Appl.\ Probab 42, 761-794, (2010). 



 \bibitem{SC}  L. Saloff-Coste,
\textit{ Lectures on finite Markov chains.} IHP Course 98,  Ecole d' Ete  de Probabilites de Saint-Flour XXVI, Lecture Notes in
         Math. 1665, 301-413, Springer (1996).




\bibitem{Ta1}  M. Talagrand,
\textit{ Concentration of measure and isoperimetric inequalities in product spaces.} Publ. Math. I.H.E.S. 81, 73-205, Springer (1995).
         
     \bibitem{Ta2}  M. Talagrand,
\textit{A new isoperimetric inequality and concentration of measure phenomenon.} In: Lindenstrauss, J., Milman, V.D. (eds.) Geometric Aspects of Functional Analysis, Lecture notes in Math. 1469, 94-124, Springer-Verlag, Berlin (1991). 



 \bibitem{W-Y} F-Y. Wang and C. Yuan, \textit{Poincar\'e inequality on the path space of Poisson point processes}. J Theor Probab 23 (3), 824-833 (2010).         

 
 \bibitem{Yo} K. Yosida, \textit{Functional Analysis}. Springer-Verlang (1980).     

 
 
\end{thebibliography}
\end{document}